\documentclass[12pt,a4paper,reqno]{amsart}
\usepackage[a4paper,left=30mm,right=30mm,top=36mm,bottom=36mm]{geometry}
\newtheorem{dfntn}{Definition}[section] 
\newtheorem{rmrk}{Remark}[section]    
\newtheorem{thrm}{Theorem}[section]     
\newtheorem{lmm}{Lemma}[section]	

\usepackage{color}

\usepackage{amssymb}
\usepackage{amsmath}
\usepackage{graphicx}
\usepackage{float}
\usepackage[utf8]{inputenc}
\usepackage[english]{babel}
\usepackage{subfig}
\usepackage{mathtools}

\newcommand{\R}{\mathbb{R}}
\newcommand{\T}{\mathbb{T}}
\newcommand{\N}{\mathbb{N}}
\newcommand{\eps}{\varepsilon}

\newcommand{\weak}{\rightharpoonup}

\renewcommand{\div}{\mathrm{div}}
\newcommand{\vertiii}[1]{{\left\vert\kern-0.25ex\left\vert\kern-0.25ex\left\vert #1
    \right\vert\kern-0.25ex\right\vert\kern-0.25ex\right\vert}}

\def\Xint#1{\mathchoice
   {\XXint\displaystyle\textstyle{#1}}%
   {\XXint\textstyle\scriptstyle{#1}}%
   {\XXint\scriptstyle\scriptscriptstyle{#1}}%
   {\XXint\scriptscriptstyle\scriptscriptstyle{#1}}%
   \!\int}
\def\XXint#1#2#3{{\setbox0=\hbox{$#1{#2#3}{\int}$}
     \vcenter{\hbox{$#2#3$}}\kern-.5\wd0}}
\def\meanint{\Xint-}

\numberwithin{equation}{section}

\usepackage{hyperref}

\begin{document}

\title[Analytical and numerical homogenization of Isaacs equations]{\bfseries\boldmath Analytical and numerical periodic homogenization of Hamilton--Jacobi--Bellman and Isaacs equations in a $W^{2,p}$ Cordes framework}

\author[T. Sprekeler]{Timo Sprekeler}
\address[Timo Sprekeler]{Department of Mathematics, Texas A{\&}M University, College Station, TX 77843, USA.}
\email{timo.sprekeler@tamu.edu}

\author[H. V. Tran]{Hung V. Tran}
\address[Hung V. Tran]{Department of Mathematics, University of Wisconsin Madison, Van Vleck Hall, 480 Lincoln Drive, Madison, Wisconsin 53706, USA.}
\email{hung@math.wisc.edu}

\date{\today}

\begin{abstract}
We study periodic homogenization of fully-nonlinear second-order Hamilton--Jacobi--Bellman and Isaacs equations subject to a homogeneous Dirichlet boundary condition and the numerical approximation of the effective Hamiltonian, where the diffusion matrix is assumed to satisfy the Cordes condition. In the first part, assuming the Cordes parameter is sufficiently large, we show existence and uniqueness of $W^{2,p}$ solutions to the true cell problem, prove homogenization under minimal regularity assumptions, and establish a quantitative homogenization result. In the second part, we propose and rigorously analyze $W^{2,p}$-conforming and mixed finite element methods for the approximation of the corrector, which are then used in the approximation of the effective Hamiltonian. The key tool used throughout this work is Campanato's concept of near operators.

\end{abstract}

\maketitle

\section{Introduction}

In this paper, we discuss the homogenization of the fully-nonlinear Isaacs problem
\begin{align}\label{Isaacs intro}
    \begin{split}
    \inf_{\alpha\in \mathcal{A}} \sup_{\beta\in \mathcal{B}} \left\{-A\left(\,\frac{\cdot}{\eps}\,,\alpha,\beta\right):D^2 u_{\eps} - f(\,\cdot\,,\alpha,\beta)\right\} &= 0\quad\text{in }\Omega,\\
    u_{\eps} &= 0\quad\text{on }\partial\Omega,
\end{split}
\end{align}
where $\eps>0$ is a parameter, $\mathcal{A}$ and $\mathcal{B}$ are compact metric spaces, $\Omega\subset \R^n$ is a bounded Lipschitz domain, $f\in C(\overline{\Omega}\times \mathcal{A}\times \mathcal{B})$, and $A = A(y,\alpha,\beta)\in C(\mathbb{R}^n\times \mathcal{A}\times \mathcal{B};\R^{n\times n}_{\mathrm{sym}})$ is $\mathbb{Z}^n$-periodic in $y$, uniformly elliptic, and satisfies the Cordes condition. When $\Omega$ is convex, it is known that there exists a unique strong solution $u_\eps\in H^2(\Omega)\cap H^1_0(\Omega)$ to \eqref{Isaacs intro}; see \cite{KS21}. Assuming that $\Omega$ is such that the Laplacian is a bijection from $W^{2,p}(\Omega)\cap W^{1,p}_0(\Omega)$ to $L^{p}(\Omega)$ for some $p>n$, and that the Cordes parameter is sufficiently large (see Definition \ref{Def: Mdelta}), we show that $u_\eps\in W^{2,p}(\Omega)\cap W^{1,p}_0(\Omega)$ for some $p>n$ and we study analytical and numerical homogenization of \eqref{Isaacs intro}.

The Cordes condition makes it possible to convert \eqref{Isaacs intro} into an equivalent problem
\begin{align*}
    H\left(\,\cdot\,,\,\frac{\cdot}{\eps}\,,D^2 u_\eps\right) := \inf_{\alpha\in \mathcal{A}} \sup_{\beta\in \mathcal{B}} \left\{-\tilde A\left(\,\frac{\cdot}{\eps}\,,\alpha,\beta\right):D^2 u_{\eps} - \tilde f\left(\,\cdot\,,\,\frac{\cdot}{\eps}\, ,\alpha,\beta\right)\right\} &= 0\quad\text{in }\Omega,\\
    u_{\eps} &= 0\quad\text{on }\partial\Omega,
\end{align*}
with $\lvert \tilde A - I_n\rvert < 1$ in $\R^n\times \mathcal{A}\times \mathcal{B}$, where the new coefficients $\tilde A$ and $\tilde f$ are obtained by multiplying $A$ and $f$ by the positive function $\lvert A\rvert^{-2}\mathrm{tr}(A)$, respectively.

Homogenization is classical and well-understood if the diffusion coefficient $A = A(y,\alpha,\beta)$ is H\"{o}lder continuous in $y$ with exponent in $(\frac{1}{2},1]$, uniformly in $\alpha$ and $\beta$; see, e.g., \cite{AB10,ABM07,AS14,CSW05,Eva92}. In this case, it is known that as $\eps\searrow 0$ we have that $(u_\eps)_{\eps>0}$ converges uniformly to the viscosity solution $u\in C(\overline{\Omega})$ of the homogenized problem 
\begin{align}\label{eff prob intro}
\begin{split}
\overline{H}(\,\cdot\,,D^2 u) &= 0\quad\text{in }\Omega,\\
    u &= 0\quad\text{on }\partial\Omega,
\end{split}
\end{align}
where the effective Hamiltonian $\overline{H}:\overline{\Omega}\times \R^{n\times n}_{\mathrm{sym}}\rightarrow \R$ is defined as follows: for each $(x,X)\in \overline{\Omega}\times \R^{n\times n}_{\mathrm{sym}}$, the value $ \overline{H}(x,X)\in \R$ is the unique constant for which the true cell problem
\begin{align}\label{cell intro}
    H(x,\,\cdot\,,X+D^2 v(\,\cdot\,;x,X)) = \overline{H}(x,X)\quad\text{in }\T^n,\qquad \int_{\T^n} v(y;x,X)\,\mathrm{d}y = 0
\end{align}
has a viscosity solution $v(\,\cdot\,;x,X)\in C(\T^n)$, referred to as a corrector. When the Hamiltonian is independent of the slow variable, an algebraic convergence rate has been established in \cite{CS10} under an additional regularity assumption on the solution to the homogenized problem. We refer to \cite{CM09,KL16,Mar14} for convergence rates in the convex or concave setting, and to \cite{AFL22,ASS14,GST25,GST27,GT25,GTY20,ST21} for convergence rates in the linear setting. In practice, problems of the form \eqref{Isaacs intro} arise in stochastic control and differential games in environments with small-scale heterogeneity. When $\mathcal{A}$ or $\mathcal{B}$ is a singleton set, the Isaacs equation reduces to the Hamilton--Jacobi--Bellman (HJB) equation. Further, the Monge--Amp\`{e}re (MA) equation arising in illumination optics and optimal transport admits a degenerate Isaacs formulation; see, e.g., \cite{BK21,FJ17}. We refer to \cite{CGL26,JLT26} for recent results in quantitative homogenization of MA equations.  

The goal of our paper is twofold: to prove qualitative and quantitative periodic homogenization results for \eqref{Isaacs intro} under the Cordes condition, and to study numerical homogenization via a finite element method.

In the first part, we will establish existence and uniqueness of a $W^{2,p}(\T^n)$ solution to the true cell problem \eqref{cell intro} using Campanato's theory of near operators, and use it to prove qualitative and quantitative periodic homogenization results for \eqref{Isaacs intro} under the Cordes condition.

The second part of this work is devoted to numerical homogenization. Motivated by the fact that, when $\eps > 0$ is small, an accurate direct simulation of \eqref{Isaacs intro} via a standard finite element method requires a prohibitively fine mesh for the accurate approximation of the solution, a classical idea is to obtain the macroscopic profile of the solution to \eqref{Isaacs intro} by approximating the solution to the homogenized problem \eqref{eff prob intro}. The main difficulty lies in the numerical approximation of the effective Hamiltonian, and we refer to \cite{GSS21} (for HJB) and \cite{KS22} (for Isaacs) for prior approaches under stronger regularity assumptions that are based on mixed and discontinuous Galerkin finite element methods for the numerical approximation of so-called approximate correctors $v^{\sigma}(\cdot;x,X)$, i.e., the solution to an approximate cell problem (cell $\sigma$-problem). A distinctive feature of the present work is that we can detour the use of approximate cell problems altogether.

The literature on the numerical approximation of effective Hamiltonians in the homogenization of fully-nonlinear second-order HJB and Isaacs equations is quite sparse. Besides \cite{GSS21,KS22}, a finite difference approximation in the HJB setting has been proposed in \cite{CM09} and some exact formulas and numerical simulations for certain types of HJB Hamiltonians can be found in \cite{FO18,FOO18}. For numerical homogenization of linear nondivergence-form equations, we refer the reader to \cite{CSS20,Spr24,Spr26,SSZ25,SWZ26} for the numerical approximation of effective diffusion matrices and to \cite{FGP24,HMS26} for numerical homogenization beyond periodicity.

On the other hand, the literature on numerical effective Hamiltonians in the homogenization of first-order Hamilton--Jacobi and Hamilton--Jacobi--Isaacs equations, and viscous Hamilton--Jacobi equations, is comparatively vast; see, e.g., \cite{ACC08,FR08,GLQ18,GO04,LYZ11,Qia03,QST24}.

This paper is structured as follows. In Section \ref{Sec: 2}, we begin by presenting the Cordes setting in Section \ref{Sec: 2.1}, and by recalling Campanato's notion of near operators in Lemma \ref{Lmm: near} which is central to essentially all arguments in this paper. We then show existence and uniqueness of a solution $u_\eps\in W^{2,p}(\Omega)\cap W^{1,p}_0(\Omega)$ to \eqref{Isaacs intro} for some $p>n$, and derive a uniform $W^{2,p}$-bound in Section \ref{Sec: 2.2}. 

In Section \ref{Sec: 3}, we study homogenization of \eqref{Isaacs intro}. After outlining the objective in Section \ref{Sec: setting hom}, we prove existence and uniqueness of a solution $v(\cdot;x,X)\in W^{2,p}(\T^n)\cap L^p_0(\T^n)$ to the true cell problem \eqref{cell intro} in Section \ref{Sec: 3.2}. In Section \ref{Sec: 3.3}, we show that $(u_\eps)_{\eps>0}$ converges weakly in $W^{2,p}(\Omega)$ to the unique solution $u\in W^{2,p}(\Omega)\cap W^{1,p}_0(\Omega)$ to the homogenized problem \eqref{eff prob intro}, and we quantify the rate of convergence in the $L^\infty(\Omega)$-norm in Section \ref{Sec: 3.4}. 

In Section \ref{Sec: 4}, we study numerical homogenization of \eqref{Isaacs intro}. We construct and rigorously analyze a $W^{2,p}$-conforming and a mixed finite element method for the approximation of the corrector in Sections \ref{Sec: 4.1} and \ref{Sec: 4.2}, respectively, which is then used to construct and analyze a method for the numerical approximation of the effective Hamiltonian in Section \ref{Sec: 4.3}.

\section{Cordes Framework for $W^{2,n}$ Solutions}\label{Sec: 2}

\subsection{Setting}\label{Sec: 2.1}
Let $\Omega \subset \R^n$ be a bounded Lipschitz domain which is assumed to be such that
\begin{align}\label{Om ass}
    -\Delta:W^{2,n}(\Omega)\cap W^{1,n}_0(\Omega)\rightarrow L^n(\Omega)\quad\text{is a bijection}.
\end{align}
Let $\mathcal{A}$ and $\mathcal{B}$ be compact metric spaces, and let $f\in C(\overline{\Omega}\times \mathcal{A}\times \mathcal{B})$. For a small parameter $\eps>0$, we consider the Isaacs problem
\begin{align}\label{ueps prob}
\begin{split}
    \inf_{\alpha\in \mathcal{A}} \sup_{\beta\in \mathcal{B}} \{-A^{\eps}_{\alpha\beta}:D^2 u_{\eps} - f_{\alpha\beta}\} &= 0\quad\text{in }\Omega,\\
    u_{\eps} &= 0\quad\text{on }\partial\Omega,
\end{split}
\end{align}
where we have written $A^{\eps}_{\alpha\beta}:= A_{\alpha\beta}\left(\frac{\cdot}{\eps}\right)$ with $A_{\alpha\beta} := A(\,\cdot\,,\alpha,\beta)$ and $f_{\alpha\beta} := f(\,\cdot\,,\alpha,\beta)$. Setting $\delta_0:= 1-c_{n}^2$, we assume that there exists a constant $\delta\in (\delta_0,1]$ such that $A\in \mathcal{M}_\delta$. Here, writing $W(\Omega):=W^{2,n}(\Omega)\cap W^{1,n}_0(\Omega)$, the constant $c_n \in (0,1]$ (see Lemma \ref{Lmm: cp leq 1}) is given by
\begin{align*}
    c_n := \|(-\Delta)^{-1}\|_{(L^n(\Omega),\|\cdot\|_{L^n(\Omega)})\rightarrow (W(\Omega),\lvert  \cdot \rvert_{W^{2,n}(\Omega)})}^{-1}  = \inf_{v\in W(\Omega)\backslash\{0\}} \frac{\|\Delta v\|_{L^n(\Omega)}}{\|D^2 v\|_{L^n(\Omega)}},
\end{align*}
and the set $\mathcal{M}_\delta$ is defined as follows: 
\begin{dfntn}\label{Def: Mdelta}
    For $\delta\in (\delta_0,1]$, we define $\mathcal{M}_\delta$ to be the set of all functions $A\in C(\mathbb{R}^n\times \mathcal{A}\times \mathcal{B};\R^{n\times n}_{\mathrm{sym}})$ that have the following properties: 
    \begin{itemize}
        \item periodicity: the map $A_{\alpha\beta}:\R^n\rightarrow \R^{n\times n}_{\mathrm{sym}}$ is $\mathbb{Z}^n$-periodic for all $(\alpha,\beta)\in \mathcal{A}\times \mathcal{B}$;
        \item uniform ellipticity: $\inf_{(y,\alpha,\beta)\in\mathbb{R}^n\times \mathcal{A}\times \mathcal{B}}\; \inf_{\xi\in \R^n\backslash\{0\}} \lvert \xi\rvert^{-2} A(y,\alpha,\beta)\xi\cdot \xi > 0$;
        \item Cordes property:
        \begin{align}\label{Cor}
        \sup_{\mathbb{R}^n\times \mathcal{A}\times \mathcal{B}}\frac{\lvert A\rvert^2}{(\mathrm{tr}(A))^2} \leq \frac{1}{n-1 + \delta},\quad\text{where}\quad \lvert A\rvert := \sqrt{A:A}.
        \end{align}
    \end{itemize}
\end{dfntn}
To simplify presentation, we regard a $\mathbb{Z}^n$-periodic function on $\R^n$ as a function on $\mathbb{T}^n := \R^n / \mathbb{Z}^n$ and vice versa.
Introducing the positive renormalization function 
\begin{align*}
    \gamma := \frac{\mathrm{tr}(A)}{\lvert A\rvert^2}\in C(\mathbb{R}^n\times \mathcal{A}\times \mathcal{B}),
\end{align*}
the key consequence of the Cordes property \eqref{Cor} is that
\begin{align}\label{gam A}
    \lvert I_n - \gamma A \rvert^2 = n -\gamma\, \mathrm{tr}(A) \leq 1-\delta\quad \text{in }\mathbb{R}^n\times \mathcal{A}\times \mathcal{B}.
\end{align}
Let us note that, when $n = 2$ and $\Omega$ is convex, then \eqref{Om ass} is satisfied and we have $c_n = 1$. Further, in this case, uniform ellipticity implies that the Cordes condition is satisfied, i.e., that there exists $\delta\in (0,1]$ such that \eqref{Cor} holds; see, e.g., \cite{SS14}.

\subsection{Existence, uniqueness, and properties of {\boldmath $W^{2,n}$} solutions}\label{Sec: 2.2}

Our goal is to prove that the operator
\begin{align}\label{op Feps}
    F_\eps : W(\Omega)\rightarrow L^n(\Omega),\quad v\mapsto \inf_{\alpha\in \mathcal{A}} \sup_{\beta\in \mathcal{B}} \{-A^{\eps}_{\alpha\beta}:D^2 v - f_{\alpha\beta}\} 
\end{align}
is bijective and to obtain a uniform (in $\eps$) bound on the $W^{2,n}$-solution to \eqref{ueps prob}.

As a starting point, we recall Campanato's concept of near operators.

\begin{lmm}[nearness of operators \cite{Cam94}]\label{Lmm: near}
Let $F,L:X \rightarrow Y$ be two maps from a nonempty set $X$ to a real Banach space $(Y,\|\cdot\|)$. Suppose that $L$ is bijective and that $F$ is near $L$, that is, there exist constants $a > 0$ and $K\in [0,1)$ such that
\begin{align*}
\| L(v_1) - L(v_2) - a\left[ F(v_1) - F(v_2)\right]\| \leq K \|L(v_1) - L(v_2)\|\quad \forall\, v_1,v_2\in X.
\end{align*}
Then, $F$ is bijective and we have the bound
\begin{align*}
\|L(v_1)-L(v_2)\| \leq \frac{a}{1-K} \|F(v_1)-F(v_2)\|\quad \forall\, v_1,v_2\in X.
\end{align*}
\end{lmm}
The above lemma is a direct consequence of Banach's fixed point theorem; see, e.g., \cite{HMS26} for a short proof. In view of \eqref{gam A}, we consider the renormalized operator
\begin{align}\label{ren}
\tilde F_\eps : W(\Omega)\rightarrow L^n(\Omega),\quad v\mapsto \inf_{\alpha\in \mathcal{A}} \sup_{\beta\in \mathcal{B}} \{-\gamma^{\eps}_{\alpha\beta}A^{\eps}_{\alpha\beta}:D^2 v - \gamma^{\eps}_{\alpha\beta}f_{\alpha\beta}\}
\end{align}
which we think of as a perturbed negative Laplacian, i.e., $\tilde F_\eps = -\Delta + T_\eps$, where
\begin{align*}
    T_\eps: W(\Omega)\rightarrow L^n(\Omega),\quad v\mapsto \inf_{\alpha\in \mathcal{A}} \sup_{\beta\in \mathcal{B}} \{(I_n-\gamma^{\eps}_{\alpha\beta}A^{\eps}_{\alpha\beta}):D^2 v - \gamma^{\eps}_{\alpha\beta}f_{\alpha\beta}\}.
\end{align*}
We can show that the perturbation $T_\eps$ is small in the following sense.
\begin{lmm}[nearness to the negative Laplacian]\label{Lmm: near Delta}
    Let $\Omega \subset \R^n$ be a bounded Lipschitz domain satisfying \eqref{Om ass}, let $\mathcal{A},\mathcal{B}$ be compact metric spaces, $f\in C(\overline{\Omega}\times \mathcal{A}\times \mathcal{B})$, and $A\in M_\delta$ for some $\delta\in (\delta_0,1]$. Further, let $\tilde F_\eps:W(\Omega)\rightarrow L^n(\Omega)$ be the operator defined in \eqref{ren}. Then, for any $v_1,v_2\in W(\Omega)$, we have the bound
    \begin{align}\label{near to neg Lap}
        \|-\Delta(v_1-v_2) - [\tilde F_\eps(v_1)-\tilde F_\eps(v_2)]\|_{L^n(\Omega)}\leq c_n^{-1}\sqrt{1-\delta}\, \|\Delta (v_1-v_2)\|_{L^n(\Omega)}. 
    \end{align}
    In particular, $\tilde F_\eps$ is near $-\Delta:W(\Omega)\rightarrow L^n(\Omega)$.
\end{lmm}
\begin{proof}
    Let $v_1,v_2\in W(\Omega)$. Arguing as in the proof of \cite[Lemma 2.4]{KS21}, we obtain
    \begin{align*}
        \lvert -\Delta(v_1-v_2) &- [\tilde F_\eps(v_1)-\tilde F_\eps(v_2)]\rvert  = \left\lvert T_\eps(v_1) - T_\eps(v_2)  \right\rvert \\
        &\leq \sup_{(\alpha,\beta)\in \mathcal{A}\times \mathcal{B}} \left\lvert   (I_n-\gamma^{\eps}_{\alpha\beta}A^{\eps}_{\alpha\beta}):D^2 (v_1-v_2)\right\rvert \leq \sqrt{1-\delta}\, \lvert D^2(v_1-v_2)\rvert
    \end{align*}
    almost everywhere in $\Omega$. In combination with the Miranda--Talenti-type bound
    \begin{align}\label{MT}
        \|D^2 v\|_{L^n(\Omega)}\leq c_n^{-1} \|\Delta v\|_{L^n(\Omega)}\qquad \forall\, v\in W(\Omega)
    \end{align}
    and the observation that $c_n^{-1}\sqrt{1-\delta} < 1$ (as $\delta > \delta_0 = 1-c_n^2$), we obtain \eqref{near to neg Lap}.
\end{proof}

We are now in a position to prove the main result of this section.
\begin{thrm}[$W^{2,n}$-solvability of Isaacs equations]\label{Thm: W2n}
    Let $\Omega \subset \R^n$ be a bounded Lipschitz domain satisfying \eqref{Om ass}, let $\mathcal{A},\mathcal{B}$ be compact metric spaces, $f\in C(\overline{\Omega}\times \mathcal{A}\times \mathcal{B})$, and $A\in M_\delta$ for some $\delta\in (\delta_0,1]$. Then, there exists a unique solution $u_\eps\in W(\Omega)$ to the Isaacs problem \eqref{ueps prob} and we have the bound
    \begin{align*}
        \sup_{\eps > 0}\;\lvert  u_\eps\rvert_{W^{2,n}(\Omega)} \leq \frac{\|\gamma\|_{L^\infty(\mathbb{R}^n\times \mathcal{A}\times \mathcal{B})}}{c_n-\sqrt{1-\delta}}\,\left\| \sup_{(\alpha,\beta)\in \mathcal{A}\times \mathcal{B}} \lvert f_{\alpha\beta}\rvert \right\|_{L^n(\Omega)}.
    \end{align*}
\end{thrm}

\begin{proof}
    In view of \eqref{Om ass}, we can combine Lemma \ref{Lmm: near} and Lemma \ref{Lmm: near Delta} to obtain that $\tilde F_\eps:W(\Omega)\rightarrow L^n(\Omega)$ defined in \eqref{ren} is a bijection and we have that
\begin{align}\label{tilde bd}
    \|\Delta(v_1-v_2)\|_{L^n(\Omega)} \leq \frac{\|\tilde F_\eps(v_1) - \tilde F_\eps(v_2)\|_{L^n(\Omega)}}{1-c_n^{-1}\sqrt{1-\delta}}
\end{align}
for any $v_1,v_2\in W(\Omega)$. Since $\gamma$ is positive and $\mathcal{A}$ and $\mathcal{B}$ are compact metric spaces, it follows from the proof of \cite[Lemma 2.2]{KS21} that $F_\eps:W(\Omega)\rightarrow L^n(\Omega)$ defined in \eqref{op Feps} is a bijection. Hence, the problem \eqref{ueps prob} has a unique solution $u_\eps\in W(\Omega)$. Using \eqref{tilde bd} and the fact that $F_\eps(u_\eps) = 0$ a.e. in $\Omega$ implies $\tilde F_\eps(u_\eps) = 0$ a.e. in $\Omega$,
    \begin{align*}
        \left(1-c_n^{-1}\sqrt{1-\delta} \right)\|\Delta u_{\eps}\|_{L^n(\Omega)} \leq \|\tilde F_\eps(0) - \tilde F_\eps(u_{\eps})\|_{L^n(\Omega)}= \|\tilde F_\eps(0)\|_{L^n(\Omega)}.
    \end{align*}
    Using the definition of $\tilde F_\eps$ and the fact that $\gamma\in L^\infty(\mathbb{R}^n\times \mathcal{A}\times \mathcal{B})$, we obtain that
    \begin{align*}
       c_n\|D^2 u_{\eps}\|_{L^n(\Omega)} \leq  \|\Delta u_{\eps}\|_{L^n(\Omega)}\leq \frac{\|\gamma\|_{L^\infty(\mathbb{R}^n\times \mathcal{A}\times \mathcal{B})}}{1-c_n^{-1}\sqrt{1-\delta}}\left\| \sup_{(\alpha,\beta)\in \mathcal{A}\times \mathcal{B}} \lvert f_{\alpha\beta}\rvert \right\|_{L^n(\Omega)},
    \end{align*}
    where we have used \eqref{MT} in the first inequality. 
\end{proof}

We note that $\lvert\, \cdot\, \rvert_{W^{2,n}(\Omega)} := \|D^2(\cdot)\|_{L^n(\Omega)}$ defines a norm on $W(\Omega)=W^{2,n}(\Omega)\cap W^{1,n}_0(\Omega)$. We conclude with two remarks.

\begin{rmrk}
    The unique solution $u_\eps\in W(\Omega)$ to \eqref{ueps prob} given by Theorem \ref{Thm: W2n} is a viscosity solution to \eqref{ueps prob}; see \cite{CCK96}.
\end{rmrk}

\begin{rmrk}\label{Rk: C1theta}
    Let the setting be as in Theorem \ref{Thm: W2n}, and suppose that, for some $s>0$, the map $-\Delta:W^{2,n+s}(\Omega)\cap W^{1,n+s}_0(\Omega)\rightarrow L^{n+s}(\Omega)$ is a bijection. For $p\in [1,\infty)$, set
    \begin{align}\label{c of p defn}
        c(p) :=  \inf_{v\in W^{2,p}(\Omega)\cap W^{1,p}_0(\Omega)\backslash\{0\}} \frac{\|\Delta v\|_{L^p(\Omega)}}{\|D^2 v\|_{L^p(\Omega)}},
    \end{align}
    and recall that $c(n) = c_n > \sqrt{1-\delta}$. Then, if $s>0$ is sufficiently small, $c(n+s) > \sqrt{1-\delta}$. Similarly to Lemma \ref{Lmm: near Delta} and Theorem \ref{Thm: W2n}, we obtain $\tilde F_\eps$ is near $-\Delta$ as maps $W^{2,n+s}(\Omega)\cap W^{1,n+s}_0(\Omega)\rightarrow L^{n+s}(\Omega)$, and the solution $u_\eps$ to \eqref{ueps prob} given by Theorem \ref{Thm: W2n} satisfies $u_\eps\in W^{2,n+s}(\Omega)$ and
    \begin{align}\label{unf w2pbd}
        \sup_{\eps > 0}\;\lvert  u_\eps\rvert_{W^{2,n+s}(\Omega)} \leq \frac{\|\gamma\|_{L^\infty(\mathbb{R}^n\times \mathcal{A}\times \mathcal{B})}}{c(n+s)-\sqrt{1-\delta}}\,\left\| \sup_{(\alpha,\beta)\in \mathcal{A}\times \mathcal{B}} \lvert f_{\alpha\beta}\rvert \right\|_{L^{n+s}(\Omega)}.
    \end{align}
    In particular, $(u_{\eps})_{\eps > 0}$ is uniformly bounded in $W^{2,n+s}(\Omega)$, and hence, also in $C^{1,\theta}(\overline{\Omega})$ for $\theta\in (0,\frac{s}{n+s})$.
\end{rmrk}

\section{Homogenization}\label{Sec: 3}

\subsection{Setting and objective}\label{Sec: setting hom}

Let the setting be as in Remark \ref{Rk: C1theta}. Then, the Isaacs problem \eqref{ueps prob} has a unique solution $u_{\eps}\in W^{2,p}(\Omega)\cap W^{1,p}_0(\Omega)$ and $(u_\eps)_{\eps > 0}$ is uniformly bounded in $W^{2,p}(\Omega)$, where
\begin{align*}
    p>n.
\end{align*}
Then, there exists $u\in W^{2,p}(\Omega)\cap W^{1,p}_0(\Omega)$ such that, upon passing to a subsequence,
\begin{align}\label{conv ueps}
\begin{split}
    u_\eps \weak u\;\;\text{weakly in }W^{2,p}(\Omega),\quad
    u_\eps \rightarrow u\;\;\text{strongly in }C^{1,\theta}(\overline{\Omega})\quad\text{ as }\eps\searrow 0
    \end{split}
\end{align}
for $\theta\in (0,1-\frac{n}{p})$. As our starting point, we use the observation that $u_\eps$ is the unique solution in $W^{2,p}(\Omega)\cap W^{1,p}_0(\Omega)$ to the renormalized Isaacs problem
\begin{align}\label{H is zero}
\begin{split}
    H\left(\,\cdot\,,\,\frac{\cdot}{\eps}\,,D^2 u_\eps\right) &= 0\quad\text{in }\Omega,\\
    u_\eps &= 0\quad\text{on }\partial\Omega,
\end{split}
\end{align}
where the Hamiltonian $H$ is given by
\begin{align*}
   H:\overline{\Omega}\times \R^n\times \R^{n\times n}_{\mathrm{sym}}\rightarrow \R,\quad H(x,y,X) := \inf_{\alpha\in \mathcal{A}} \sup_{\beta\in \mathcal{B}} \{\gamma_{\alpha\beta}(y)\left(-A_{\alpha\beta}(y):X - f_{\alpha\beta}(x)\right)\}.
\end{align*}
We aim to show that $u$ is the unique solution in $W^{2,p}(\Omega)\cap W^{1,p}_0(\Omega)$ to the problem
\begin{align*}
\overline{H}(\,\cdot\,,D^2 u) &= 0\quad\text{in }\Omega,\\
    u &= 0\quad\text{on }\partial\Omega
\end{align*}
with an appropriately defined effective Hamiltonian $\overline{H}:\overline{\Omega}\times \R^{n\times n}_{\mathrm{sym}}\rightarrow \R$, and to quantify the rate of convergence. Moreover, we would like to construct the effective Hamiltonian without making any additional regularity assumptions, and in a way that reveals a natural finite element framework for its numerical approximation.

\subsection{Effective Hamiltonian via true cell problem}\label{Sec: 3.2}

Writing $L^p_0(\mathbb{T}^n) := \{v\in L^p(\mathbb{T}^n):\int_{\mathbb{T}^n} v = 0\}$, we can prove the existence and uniqueness of $W^{2,p}(\mathbb{T}^n)\cap L^p_0(\mathbb{T}^n)$ solutions to the true cell problem provided that the Cordes parameter $\delta$ is sufficiently large. For the quantification of how large the Cordes parameter needs to be, the following quantity shall be helpful:
\begin{align}\label{xi of p}
    \xi(p) := \inf_{\eta\in L^p_0(\T^n)\backslash\{0\}} \frac{\inf_{r\in \mathbb{P}_0(\T^n)} \|\eta - r\|_{L^p(\T^n)}}{\|\eta\|_{L^p(\T^n)}}.
\end{align}
Here, $\mathbb{P}_0(\T^n)$ denotes the set of constant real-valued functions on $\T^n$.
\begin{lmm}[explicit formula for $\xi(p)$]\label{Lmm: xi expl form}
    Let $\xi(p)$ denote the constant defined in \eqref{xi of p} and let $q := \frac{p}{p-1}$ denote the dual exponent. Then,
    \begin{align}\label{xi formula}
    \xi(p) = \left\| g\mapsto g-\int_{\T^n} g\right\|_{L^p(\T^n) \rightarrow L^p(\T^n)}^{-1}=\inf_{t\in (0,\infty)} \frac{1+t}{(1+t^{p-1})^{\frac{1}{p}}(1+t^{q-1})^{\frac{1}{q}}}.
    \end{align}
    In particular, $\xi(p) = \xi(q) \in \left[2^{-\lvert 1 - \frac{2}{p}\rvert},1\right]\subset (\frac{1}{2},1]$. Some explicit values are $\xi(2) = 1$, $\xi(3) = 0.91\ldots$, and $\xi(4) =   0.82\ldots$.
\end{lmm}

\begin{proof}
Introducing $T:L^p(\T^n) \rightarrow L^p(\T^n)$ given by $T(g):=g - \int_{\T^n} g$, we have that
\begin{align*}
    (\xi(p))^{-1} &= \sup_{\eta\in L^p_0(\T^n)\backslash\{0\}} \frac{\|\eta\|_{L^p(\T^n)}}{\inf_{r\in \mathbb{P}_0(\T^n)} \|\eta - r\|_{L^p(\T^n)}} \\
    &= \sup_{\eta\in L^p_0(\T^n)\backslash\{0\}} \sup_{\substack{
    g\in L^p(\T^n)\backslash\{0\}\\
    T(g)=\eta
    }} \frac{\|\eta\|_{L^p(\T^n)}}{\|g\|_{L^p(\T^n)}} 
    = \sup_{g\in L^p(\T^n)\backslash\{0\}} \frac{\|T(g)\|_{L^p(\T^n)}}{\|g\|_{L^p(\T^n)}},
\end{align*}
which yields the first equality in \eqref{xi formula}. The remaining claims follow from \cite{SS23}. 
\end{proof}

Let us further introduce
\begin{align}\label{cpTn}
    c_{\T^n}(p) := \inf_{v\in W_p\backslash\{0\}} \frac{\|\Delta v\|_{L^p(\T^n)}}{\|D^2 v\|_{L^p(\T^n)}},
\end{align}
and note that $c_{\T^n}(p) \geq c(p)$; see Lemma \ref{Lmm: cp Om leq Tn} in the appendix. 
This constant is known explicitly in dimension $n = 2$:

\begin{rmrk}\label{Rk: expl c}
    In dimension $n = 2$, we have $c_{\T^2}(p) = \sqrt{ 2(1+(p-1)^2)^{-1}}$ for $p\geq 2$. Indeed, for $v\in W_p$ with $p\geq 2$, it follows from \cite{AAB26,BSV13} that 
    \begin{align*}
        \|2\lvert D^2 v\rvert^2 - (\Delta v)^2\|_{L^{\frac{p}{2}}(\T^2)} = \| (\partial_{11}^2 v - \partial_{22}^2 v)^2 + 4(\partial^2_{12}v)^2\|_{L^{\frac{p}{2}}(\T^2)} \leq (p-1)^2\|\Delta v\|_{L^p(\T^2)}^2,
    \end{align*}
    and thus, $2\|D^2 v\|_{L^p(\T^2)}^2 \leq (1+(p-1)^2)\|\Delta v\|_{L^p(\T^2)}^2$, and the constant is sharp.
\end{rmrk}

We can now state and prove the main result of this subsection.

\begin{thrm}[$W^{2,p}$-solvability of the true cell problem]\label{Thm: cell}
    Let the situation be as described in Section \ref{Sec: setting hom}. With $\xi(p)$ and $c_{\T^n}(p)$ defined in \eqref{xi of p} and \eqref{cpTn}, respectively, suppose that $\delta > 1 - \left(\xi(p)c_{\T^n}(p)\right)^2$. Then, for any pair $(x,X)\in \overline{\Omega}\times \R^{n\times n}_{\mathrm{sym}}$ there exists a unique solution $v(\,\cdot\,;x,X)\in W^{2,p}(\mathbb{T}^n)\cap L^p_0(\mathbb{T}^n)$ to the true cell problem
    \begin{align}\label{cell}
        H(x,\,\cdot\,,X+D^2 v(\,\cdot\,;x,X)) = \int_{\mathbb{T}^n} H(x,\,\cdot\,,X+D^2 v(\,\cdot\,;x,X))\quad\text{a.e. in }\mathbb{T}^n,
    \end{align}
    and we have the bound
    \begin{align}\label{v w2p bd}
        \lvert v(\,\cdot\,;x,X)\rvert_{W^{2,p}(\T^n)} \leq \frac{\| H(x,\,\cdot\,,X)\|_{L^p(\T^n)}}{\xi(p)\, c_{\T^n}(p)-\sqrt{1-\delta} }.
    \end{align}
    
\end{thrm}
\begin{proof}
Let $(x,X)\in \overline{\Omega}\times \R^{n\times n}_{\mathrm{sym}}$ be fixed, and let $q := \frac{p}{p-1}$ denote the dual exponent of $p$. For $r\in \{p,q\}$, we write $W_r := W^{2,r}(\mathbb{T}^n)\cap L^r_0(\mathbb{T}^n)$ and $\|\cdot\|_{\Delta,r} := \|\Delta(\cdot)\|_{L^r(\mathbb{T}^n)}$, and note that $(W_r,\|\cdot\|_{\Delta,r})$ is a Banach space. For $i = 1,2$, we introduce the operator 
\begin{align*}
    \mathcal{C}_i:W_p\rightarrow W_q^\ast&,\quad w\mapsto \mathcal{C}_i(w),
\end{align*}
where the actions of $\mathcal{C}_1(w)$ and $\mathcal{C}_2(w)$ are given by
\begin{align*}
    \langle \mathcal{C}_1(w),z\rangle := \int_{\T^n} (H(x,\cdot,X+D^2 w)) (-\Delta z),\qquad
    \langle \mathcal{C}_2(w),z\rangle := \int_{\T^n} (-\Delta w) (-\Delta z) 
\end{align*}
for $z\in W_q$. Note that $\mathcal{C}_2$ is linear.

\textbf{Step 1.} We show that $\mathcal{C}_1$ is near $\mathcal{C}_2$. To this end, let $v_1,v_2\in W_p$. First, we have
\begin{align}\label{corr pf St1}
\begin{split}
    &\| \mathcal{C}_2(v_1) - \mathcal{C}_2(v_2) - [\mathcal{C}_1(v_1) - \mathcal{C}_1(v_2)] \|_{W_q^\ast} \\&\qquad= \sup_{z\in W_q\backslash\{0\}} \|z\|_{\Delta,q}^{-1}\;\lvert \langle\mathcal{C}_1(v_1) - \mathcal{C}_1(v_2) -\mathcal{C}_2(v_1-v_2),z\rangle\rvert  \\
    &\qquad\leq \|H(x,\cdot,X+D^2 v_1)-H(x,\cdot,X+D^2 v_2) + \Delta(v_1-v_2)\|_{L^p(\T^n)}.
\end{split}
\end{align}
Next, we use the definition of $H$, properties of $\inf$ and $\sup$, and \eqref{gam A} to obtain that
\begin{align}\label{Lip type}
\begin{split}
    &\lvert H(x,\cdot,X+D^2 v_1)-H(x,\cdot,X+D^2 v_2) + \Delta(v_1-v_2) \rvert \\
    &= \left\lvert \inf_{\alpha\in \mathcal{A}} \sup_{\beta\in \mathcal{B}} \left\{ (I_n-\gamma_{\alpha\beta}A_{\alpha\beta}):D^2 v_1 -\gamma_{\alpha\beta}\left(A_{\alpha\beta}:X - f_{\alpha\beta}(x)\right)   \right\}\right.  \\
    &\qquad \left.-\inf_{\alpha\in \mathcal{A}} \sup_{\beta\in \mathcal{B}} \left\{ (I_n-\gamma_{\alpha\beta}A_{\alpha\beta}):D^2 v_2 -\gamma_{\alpha\beta}\left(A_{\alpha\beta}:X - f_{\alpha\beta}(x)\right)   \right\}\right\rvert \\
    &\leq \sup_{(\alpha,\beta)\in \mathcal{A}\times \mathcal{B}} \left\lvert (I_n-\gamma_{\alpha\beta}A_{\alpha\beta}):D^2 (v_1 - v_2) \right\rvert \leq \sqrt{1-\delta}\;\lvert D^2(v_1-v_2)\rvert.
\end{split}
\end{align}
In combination with \eqref{corr pf St1} and \eqref{cpTn}, we deduce that
\begin{align*}
    \| \mathcal{C}_2(v_1) - \mathcal{C}_2(v_2) - [\mathcal{C}_1(v_1) - \mathcal{C}_1(v_2)] \|_{W_q^\ast} &\leq \sqrt{1-\delta}\, \|D^2(v_1-v_2)\|_{L^p(\T^n)}\\
    &\leq \left(c_{\T^n}(p)\right)^{-1} \sqrt{1-\delta}\, \|v_1-v_2\|_{\Delta,p}.
\end{align*}
Thus, writing $\tilde K := \left(c_{\T^n}(p)\right)^{-1} \sqrt{1-\delta}$, we have the bound
\begin{align*}
    \| \mathcal{C}_2(v_1) - \mathcal{C}_2(v_2) - [\mathcal{C}_1(v_1) - \mathcal{C}_1(v_2)] \|_{W_q^\ast} \leq \tilde K\, \|v_1-v_2\|_{\Delta,p}.
\end{align*}
In view of Remark \ref{Rk: C1theta} and $c_{\T^n}(p) \geq c(p)$, we have that $\tilde K\in [0,1)$. Finally,
\begin{align*}
    \|\mathcal{C}_2(v_1) - \mathcal{C}_2(v_2)\|_{W_q^\ast} &= \sup_{z\in W_{q}\backslash\{0\}} \|z\|_{\Delta,q}^{-1} \;\left\lvert \int_{\T^n}\Delta (v_1-v_2)\,\Delta z \right\rvert \\ &= \sup_{g\in L^q_0(\T^n)\backslash\{0\}} \|g\|_{L^q(\T^n)}^{-1} \;\left\lvert \int_{\T^n}\Delta (v_1-v_2)\,g \right\rvert \\ &=  \inf_{r\in \mathbb{P}_0(\T^n)} \|\Delta(v_1-v_2) - r\|_{L^p(\T^n)} \geq \xi(p)\| v_1-v_2 \|_{\Delta,p},
\end{align*}
where we have used \cite[Proposition 11.10]{Bre11} in the third step. Altogether,
\begin{align}\label{K tilde near}
    \| \mathcal{C}_2(v_1) - \mathcal{C}_2(v_2) - [\mathcal{C}_1(v_1) - \mathcal{C}_1(v_2)] \|_{W_q^\ast} \leq \tilde K (\xi(p))^{-1}\, \|\mathcal{C}_2(v_1) - \mathcal{C}_2(v_2)\|_{W_q^\ast}.
\end{align}
Noting that $\tilde K (\xi(p))^{-1} < 1$ as $\delta > 1 - \left(\xi(p)c_{\T^n}(p)\right)^2$, we conclude that $\mathcal{C}_1$ is near $\mathcal{C}_2$ .

\textbf{Step 2.} We show that $\mathcal{C}_2$ is a bijection. Let us begin by showing injectivity. To this end, let $v_1,v_2\in W_p$ be such that $\mathcal{C}_2(v_1) = \mathcal{C}_2(v_2)$. By linearity of $\mathcal{C}_2$, we have
\begin{align*}
    \int_{\T^n} \Delta (v_1-v_2) \Delta z = \langle \mathcal{C}_2(v_1-v_2),z\rangle = 0\qquad \forall\, z\in W_q.
\end{align*}
Since $\Delta(W_q) = L^q_0(\T^n)$, we obtain $\int_{\T^n} \Delta (v_1-v_2) g = 0$ for all $g\in L^q_0(\T^n)$, and thus, $\Delta (v_1-v_2) = \int_{\T^n} \Delta (v_1-v_2) = 0$ a.e. in $\T^n$. Since $v_1-v_2\in L^p_0(\T^n)$, we conclude that $v_1 - v_2 = 0$ a.e. in $\T^n$. Hence, $\mathcal{C}_2$ is injective.

It remains to show that $\mathcal{C}_2$ is surjective. To this end, let $\Phi\in W_q^\ast$. In view of the fact that $\left.\Delta\right\rvert_{W_q}:W_q \rightarrow L^q_0(\T^n)$ is bijective, we introduce
\begin{align*}
    L: L^q(\T^n)\rightarrow \R,\qquad g\mapsto \langle L,g\rangle := \left\langle\Phi,(\left.\Delta\right\rvert_{W_q})^{-1} \left(g-\int_{\T^n}g\right)\right\rangle.
\end{align*}
Noting that $L \in (L^q(\T^n))^\ast$, we deduce that there exists $\ell \in L^p(\T^n)$ such that $\langle L,g\rangle = \int_{\T^n} \ell g$ for all $g\in L^q(\T^n)$. We have $\ell \in L^p_0(\T^n)$ since $\int_{\T^n} \ell = \langle L,1\rangle = 0$. In view of the fact that $\left.\Delta\right\rvert_{W_p}:W_p \rightarrow L^p_0(\T^n)$ is bijective, we define $w := (\left.\Delta\right\rvert_{W_p})^{-1} \ell \in W_p$. Then, there holds
\begin{align*}
    \langle\mathcal{C}_2(w),z\rangle = \int_{\T^n} \Delta w\, \Delta z = \int_{\T^n} \ell\, \Delta z = \langle L,\Delta z\rangle = \langle \Phi, z\rangle\qquad \forall\, z\in W_q,
\end{align*}
i.e., $\mathcal{C}_2(w) = \Phi$. Hence, $\mathcal{C}_2$ is surjective.

\textbf{Step 3.} In view of Lemma \ref{Lmm: near}, we combine Steps 1 and 2 to obtain that $\mathcal{C}_1$ is a bijection and we have the bound
\begin{align}\label{St 3}
    \| \mathcal{C}_2(v_1 - v_2) \|_{W_q^\ast} \leq (1-\tilde K (\xi(p))^{-1})^{-1} \| \mathcal{C}_1(v_1) - \mathcal{C}_1(v_2)\|_{W_q^\ast}\quad \forall\, v_1,v_2\in W_{p}.
\end{align}
In particular, there exists a unique $v\in W_p$ such that $\mathcal{C}_1(v) = 0$, i.e.,
\begin{align*}
    \int_{\T^n} (H(x,\cdot,X+D^2 v)) (-\Delta z) = 0\qquad \forall\, z\in W_q.
\end{align*}
Equivalently, since $\Delta(W_q) = L^q_0(\T^n)$, there exists a unique $v\in W_p$ such that $H(x,\,\cdot\,,X+D^2 v) = \int_{\mathbb{T}^n} H(x,\,\cdot\,,X+D^2 v)$ a.e. in $\T^n$, which is the cell problem \eqref{cell}. Further, choosing $v_1 = v$ and $v_2 = 0$ in \eqref{St 3}, we obtain the bound
\begin{align*}
    c_{\T^n}(p) \|D^2 v\|_{L^p(\T^n)} \leq \|\Delta v\|_{L^p(\T^n)} &\leq (\xi(p))^{-1}\|\mathcal{C}_2 (v)\|_{W_q^\ast} \\ &\leq (\xi(p))^{-1}(1-\tilde K (\xi(p))^{-1})^{-1} \| \mathcal{C}_1(0)\|_{W_q^\ast} \\ &\leq (\xi(p)-\tilde K )^{-1} \| H(x,\,\cdot\,,X)\|_{L^p(\T^n)},
\end{align*}
which yields the claimed bound \eqref{v w2p bd}.
\end{proof}

With the corrector functions $\{v(\cdot;x,X)\}_{(x,X)\in \overline{\Omega}\times \R^{n\times n}_{\mathrm{sym}}}$ from Theorem \ref{Thm: cell} at hand, we can define the effective Hamiltonian.

\begin{dfntn}[effective Hamiltonian]\label{Def: Hbar}
Let the situation be as in Theorem \ref{Thm: cell}. We define the effective Hamiltonian
\begin{align*}
  \overline{H}:\overline{\Omega}\times \R^{n\times n}_{\mathrm{sym}}\rightarrow \R,\qquad  (x,X)\mapsto\overline{H}(x,X) := \int_{\mathbb{T}^n} H(x,\,\cdot\,,X+D^2 v(\,\cdot\,;x,X)),
\end{align*}
where $\{v(\,\cdot\,;x,X)\}_{(x,X)\in \overline{\Omega}\times \R^{n\times n}_{\mathrm{sym}}}$ denote the corrector functions from Theorem \ref{Thm: cell}.
\end{dfntn}

\subsection{Convergence result}\label{Sec: 3.3}

Let us begin by showing the existence and uniqueness of a $W^{2,p}$-solution to the homogenized problem.

\begin{lmm}[$W^{2,p}$-solvability of the homogenized problem]\label{Lmm: solv hom prob}
    Let the situation be as in Theorem \ref{Thm: cell}, and let $\overline{H}$ denote the effective Hamiltonian defined in Definition \ref{Def: Hbar}. Then, the homogenized problem
    \begin{align}\label{hom prob}
    \begin{split}
\overline{H}(\,\cdot\,,D^2 u) &= 0\quad\text{in }\Omega,\\
    u &= 0\quad\text{on }\partial\Omega
    \end{split}
    \end{align}
    has a unique solution $u\in W^{2,p}(\Omega)\cap W^{1,p}_0(\Omega)$, and we have the bound
    \begin{align}\label{effbd}
        \lvert u\rvert_{W^{2,p}(\Omega)} \leq (c(p)-\sqrt{1-\delta})^{-1}  \left\| \overline{H}(\,\cdot\,,0)  \right\|_{L^p(\Omega)}.
    \end{align}
\end{lmm}

\begin{proof}
\textbf{Step 1.} We show that, for any $x\in \overline{\Omega}$ and $X_1,X_2\in \R^{n\times n}_{\mathrm{sym}}$, we have
\begin{align}\label{goal St1 hom eqn}
    \left\lvert \overline{H}(x,X_1) - \overline{H}(x,X_2) + \mathrm{tr}(X_1-X_2)\right\rvert \leq \sqrt{1-\delta}\, \lvert X_1 - X_2\rvert.
\end{align}
To this end, let $x\in \overline{\Omega}$ and $X_1,X_2\in \R^{n\times n}_{\mathrm{sym}}$. Writing $w := v(\,\cdot\,;x,X_1) - v(\,\cdot\,;x,X_2)$ and $\Theta := \overline{H}(x,X_1) - \overline{H}(x,X_2) + \mathrm{tr}(X_1-X_2)$, an argument similar to \eqref{Lip type} yields
\begin{align}\label{tointe}
\begin{split}
    \lvert \Theta + \Delta w\rvert^2 &= \lvert H(x,\cdot,X_1 + D^2 v(\cdot;x,X_1)) - H(x,\cdot,X_2 + D^2 v(\cdot;x,X_2)) \\ &\qquad+ \mathrm{tr}(X_1 - X_2 + D^2 w) \rvert^2 \\
    &\leq (1-\delta) \lvert X_1 - X_2 + D^2 w\rvert^2\quad\text{a.e. in }\T^n.
    \end{split}
\end{align}
Noting that $\int_{\T^n} \Delta w = \int_{\T^n} D^2 w = 0$, integrating inequality \eqref{tointe} over $\T^n$ yields
\begin{align*}
    \lvert \Theta\rvert^2 + \|\Delta w\|_{L^2(\T^n)}^2 \leq (1-\delta) \lvert X_1 - X_2\rvert^2 + (1-\delta)\|D^2 w\|_{L^2(\T^n)}^2  .
\end{align*}
In view of the identity $\|D^2 w\|_{L^2(\T^n)} = \|\Delta w\|_{L^2(\T^n)}$, we can absorb the final term on the right-hand side into the left-hand side to conclude that
\begin{align*}
    \lvert \Theta\rvert^2 \leq \lvert \Theta\rvert^2 + \delta\|\Delta w\|_{L^2(\T^n)}^2 \leq (1-\delta)\lvert X_1 - X_2\rvert^2,
\end{align*}
i.e., \eqref{goal St1 hom eqn} holds.

\textbf{Step 2.} We observe that the operator
\begin{align*}
    \overline{F}: W^{2,p}(\Omega)\cap W^{1,p}_0(\Omega)\rightarrow L^p(\Omega),\qquad w\mapsto \overline{H}(\,\cdot\,,D^2 w)
\end{align*}
is near $-\Delta : W^{2,p}(\Omega)\cap W^{1,p}_0(\Omega)\rightarrow L^p(\Omega)$ since, using \eqref{goal St1 hom eqn},
\begin{align*}
   \left\| \overline{F}(w_1) -  \overline{F}(w_2) + \Delta(w_1 - w_2) \right\|_{L^p(\Omega)}&\leq \sqrt{1-\delta}\,\left\| D^2 (w_1-w_2)\right\|_{L^p(\Omega)}\\
   &\leq \sqrt{1-\delta}\,(c(p))^{-1} \|\Delta (w_1-w_2)\|_{L^p(\Omega)},
\end{align*}
where $c(p) > \sqrt{1-\delta}$ is the constant from \eqref{c of p defn}. In view of Remark \ref{Rk: C1theta} and Lemma \ref{Lmm: near}, we conclude that $\overline{F}$ is bijective. In particular, there exists a unique $u\in W^{2,p}(\Omega)\cap W^{1,p}_0(\Omega)$ such that $\overline{F}(u) = 0$ a.e. in $\Omega$, and we have the bound
\begin{align*}
    c(p)\|D^2 u\|_{L^p(\Omega)} \leq \| \Delta u\|_{L^p(\Omega)} \leq \frac{\| \overline{F}(0) - \overline{F}(u) \|_{L^p(\Omega)}}{1-\sqrt{1-\delta}\,(c(p))^{-1}} = \frac{\| \overline{H}(\,\cdot\,,0)  \|_{L^p(\Omega)}}{1-\sqrt{1-\delta}\,(c(p))^{-1}},
\end{align*}
which yields \eqref{effbd}.
\end{proof}

We are now in a position to prove a qualitative homogenization result.

\begin{thrm}[homogenization of Isaacs equations in a Cordes setting]\label{Thm: Hom thm}
    Let the situation be as in Lemma \ref{Lmm: solv hom prob}. Then, as $\eps\searrow 0$, the unique solution $u_{\eps}\in W^{2,p}(\Omega)\cap W^{1,p}_0(\Omega)$ of the Isaacs problem \eqref{ueps prob} converges weakly in $W^{2,p}(\Omega)$ to the unique solution $u\in W^{2,p}(\Omega)\cap W^{1,p}_0(\Omega)$ of the homogenized problem \eqref{hom prob}.
\end{thrm}

\begin{proof}
We recall from the discussion preceding \eqref{conv ueps} that there exists a function $u\in W^{2,p}(\Omega)\cap W^{1,p}_0(\Omega)$ such that, upon passing to a subsequence,
\begin{align}\label{conv ueps 2}
\begin{split}
    u_\eps \weak u\;\;\text{weakly in }W^{2,p}(\Omega),\quad
    u_\eps \rightarrow u\;\;\text{strongly in }C^{1,\theta}(\overline{\Omega})\quad\text{ as }\eps\searrow 0
    \end{split}
\end{align}
for $\theta\in (0,1-\frac{n}{p})$. In view of Lemma \ref{Lmm: solv hom prob}, the proof is complete if we can show that $\overline{H}(\,\cdot\,,D^2 u) = 0$ a.e. in $\Omega$. In fact, we will show that $\overline{H}(x_0,D^2 u(x_0)) = 0$ for all $L^p$-Lebesgue points of $D^2 u$. To this end, let $x_0 \in \Omega$ be an $L^p$-Lebesgue point of $D^2 u$, and suppose for a proof by contradiction that $\overline{H}(x_0,D^2 u(x_0)) \neq 0$. 

First, let us note that there exists a modulus of continuity $\omega$ such that
\begin{align*}
    \lvert H(x_1,y,X) - H(x_2,y,X)\rvert \leq \omega(\lvert x_1-x_2\rvert)\qquad \forall x_1,x_2\in \overline{\Omega},\; y\in \R^n,\; X\in \R^{n\times n}_{\mathrm{sym}},
\end{align*}
and, by an argument similar to \eqref{Lip type}, we have
\begin{align}\label{HLip}
\begin{split}
    &\left\lvert H(x,y,X_1) - H(x,y,X_2) \right\rvert \\&\qquad\leq \left\lvert H(x,y,X_1) - H(x,y,X_2) + \mathrm{tr}(X_1 - X_2) \right\rvert + \sqrt{n}\, \lvert X_1 - X_2\rvert \\
    &\qquad\leq \left(\sqrt{1-\delta} + \sqrt{n}\right)\lvert X_1 - X_2\rvert \qquad \forall x\in \overline{\Omega},\; y\in \R^n,\; X_1,X_2\in \R^{n\times n}_{\mathrm{sym}}. 
\end{split}
\end{align}
Writing $q(x) := u(x_0) + \nabla u(x_0)\cdot  (x-x_0) + \frac{1}{2}(D^2 u(x_0)(x-x_0))\cdot (x-x_0)$, we use Morrey's inequality to obtain
\begin{align*}
    \frac{\|u - q\|_{L^\infty(B_r(x_0))}}{r^2} \lesssim  \frac{\| D^2(u-q) \|_{L^p(B_r(x_0))}}{r^{n/p}}  \lesssim \left(\meanint_{B_r(x_0)} \lvert D^2 u - D^2 u(x_0)\rvert^p  \right)^{1/p} \underset{r\searrow 0}{\longrightarrow} 0.
\end{align*}
Hence, we can fix $r > 0$ such that $\overline{B_r(x_0)} \subset \Omega$ and
\begin{align*}
    r^{-2}\|u-q\|_{L^\infty(B_r(x_0))}\leq \frac{ \left\lvert \overline{H}(x_0,D^2 u(x_0)) \right\rvert}{16 \sqrt{n}(\sqrt{1-\delta} + \sqrt{n})},\qquad \omega(r)\leq \frac{\left\lvert \overline{H}(x_0,D^2 u(x_0)) \right\rvert}{4}.
\end{align*}
Consider the perturbed test function $\Psi_\eps\in W^{2,p}(B_r(x_0))$ given by 
\begin{align*}
    \Psi_\eps(x) := q(x) + \eps^2 v\left(\frac{x}{\eps}\,; x_0,D^2 u(x_0)\right) + \frac{ \overline{H}(x_0,D^2 u(x_0))}{8 \sqrt{n}(\sqrt{1-\delta} + \sqrt{n})}\left(\lvert x-x_0\rvert^2-r^2\right).
\end{align*}
Using \eqref{cell} and \eqref{HLip}, we have for a.e. $x\in B_r(x_0)$ that
\begin{align*}
    &\left\lvert H\left(x,\frac{x}{\eps},D^2\Psi_\eps(x)\right) - \overline{H}(x_0,D^2 u(x_0)) \right\rvert \\&= \left\lvert H\left(x,\frac{x}{\eps},D^2\Psi_\eps(x)\right) - H\left(x_0,\frac{x}{\eps},D^2 u(x_0)+D^2 v\left(\frac{x}{\eps};x_0,D^2 u(x_0)\right)\right) \right\rvert \\
    &\leq \left\lvert H\left(x,\frac{x}{\eps},D^2\Psi_\eps(x)\right) - H\left(x,\frac{x}{\eps},D^2 u(x_0)+D^2 v\left(\frac{x}{\eps};x_0,D^2 u(x_0)\right)\right) \right\rvert + \omega(\lvert x-x_0\rvert) \\
    &\leq (\sqrt{1-\delta}\hspace{-0.05cm} +\hspace{-0.07cm} \sqrt{n}) \left\lvert  D^2\Psi_\eps(x) - D^2 u(x_0) - D^2 v\left(\frac{x}{\eps};x_0,D^2 u(x_0)\right) \right\rvert + \frac{ \lvert \overline{H}(x_0,D^2 u(x_0))\rvert }{4}\\
    &= \frac{\lvert \overline{H}(x_0,D^2 u(x_0))\rvert}{2}.
\end{align*}
Writing $s:= \mathrm{sign}(\overline{H}(x_0,D^2 u(x_0)))$ and $z_\eps:= s(u_\eps - \Psi_\eps)\in W^{2,p}(B_r(x_0))$, this implies
\begin{align*}
    a_\eps:D^2 z_\eps = sH\left(\,\cdot\,,\,\frac{\cdot}{\eps}\,,D^2\Psi_\eps\right) \geq \frac{\lvert \overline{H}(x_0,D^2 u(x_0))\rvert}{2} > 0\quad\text{a.e. in }B_r(x_0),
\end{align*}
where $a_\eps\in L^\infty(B_r(x_0);\R^{n\times n}_{\mathrm{sym}})$ is given by
\begin{align}\label{linearization}
\begin{split}
a_\eps &:= I_n-\frac{\Delta (u_\eps - \Psi_\eps)-H(\cdot,\frac{\cdot}{\eps},D^2\Psi_\eps)}{|D^2(u_\eps - \Psi_\eps)|^2}D^2(u_\eps - \Psi_\eps)\quad\text{on }\{ \lvert D^2(u_\eps - \Psi_\eps)\rvert\neq 0\},\\
a_\eps &:= I_n\qquad\qquad\qquad\qquad\qquad\qquad\qquad\qquad\qquad\quad\;\;\text{ on }\{ \lvert D^2(u_\eps - \Psi_\eps)\rvert = 0\}.
\end{split}
\end{align}
In view of \eqref{H is zero} and \eqref{Lip type}, we observe that $\lvert a_\eps - I_n\rvert \leq \sqrt{1-\delta}$ and $a_\eps \geq (1-\sqrt{1-\delta})I_n$ a.e. in $B_r(x_0)$. Hence, by the maximum principle, $z_\eps(x_0)\leq \max_{\partial B_r(x_0)} z_\eps$. Passing to the limit $\eps\searrow 0$, using \eqref{conv ueps 2}, we find that
\begin{align*}
    \frac{ \lvert \overline{H}(x_0,D^2 u(x_0))\rvert}{8 \sqrt{n}(\sqrt{1-\delta} + \sqrt{n})} r^2 \leq \max_{\partial B_r(x_0)} s(u - q) \leq \frac{ \left\lvert \overline{H}(x_0,D^2 u(x_0)) \right\rvert}{16 \sqrt{n}(\sqrt{1-\delta} + \sqrt{n})}r^2,
\end{align*}
which contradicts the assumption that $ \overline{H}(x_0,D^2 u(x_0))\neq 0$. The proof is complete. 
\end{proof}

\subsection{Convergence rate}\label{Sec: 3.4}

\begin{thrm}[$L^\infty$-rate of convergence]
    Let the situation be as in Theorem \ref{Thm: Hom thm}, and suppose that $u\in C^{2,\theta}(\overline{\Omega})$ and $f_{\alpha\beta}\in C^\theta(\overline{\Omega})$ uniformly in $(\alpha,\beta)\in \mathcal{A}\times \mathcal{B}$. Then,
    \begin{align}\label{rate}
        \|u_{\eps}-u\|_{L^\infty(\Omega)} = \mathcal{O}(\eps^\theta)\quad\text{as }\eps\searrow 0.
    \end{align}
\end{thrm}
\begin{proof}
\textbf{Step 1.} We show that there exists a constant $c>0$, depending only on $n,\|\gamma\|_\infty,p,\Omega,\delta, \sup_{\alpha,\beta}\|f_{\alpha\beta}\|_{C^{\theta}(\overline{\Omega})}$, and $\|u\|_{C^{2,\theta}(\overline{\Omega})}$, such that
    \begin{align}
        \|v(\,\cdot\,;x,D^2u(x))\|_{L^\infty(\T^n)} &\leq c \label{v prep 1}, \\
        \| v(\,\cdot\,;x_1,D^2u(x_1)) - v(\,\cdot\,;x_2,D^2u(x_2))\|_{L^\infty(\T^n)} &\leq c\lvert x_1-x_2\rvert^\theta \label{v prep 2}, \\
        \lvert u_\eps(x) - u(x) \rvert &\leq c\, \mathrm{dist}(x,\partial\Omega) \label{dist bd}
    \end{align}
    for all $x,x_1,x_2\in \overline{\Omega}$. To this end, let $x,x_1,x_2\in \overline{\Omega}$. First, note that by \eqref{v w2p bd} we have
    \begin{align*}
        \lvert v(\,\cdot\,;x,D^2u(x))\rvert_{W^{2,p}(\T^n)} \lesssim_{n,p,\delta} \| H(x,\,\cdot\,,D^2u(x))\|_{L^p(\T^n)}\lesssim_{n,\|\gamma\|_\infty,p,\delta} \lvert D^2 u(x)\rvert + \|f\|_\infty,
    \end{align*}
    which yields \eqref{v prep 1} in view of the Sobolev embeddings and $v\in L^p_0(\T^n)$. Next, writing $X_i := D^2u(x_i)$ and $v_i:=v(\,\cdot\,;x_i,X_i)$ for $i\in\{1,2\}$, we use \eqref{cell} and \eqref{hom prob} to obtain $H(x_1,\,\cdot\,,X_1 + D^2 v_1) = H(x_2,\,\cdot\,,X_2 + D^2 v_2)$ a.e. in $\T^n$. Hence, by \eqref{cpTn}, Lemma \ref{Lmm: cp Om leq Tn}, and \eqref{Lip type}, we have
    \begin{align*}
        &c_{\T^n}(p)\|D^2 (v_1-v_2)\|_{L^p(\T^n)} \leq \|\Delta (v_1-v_2)\|_{L^p(\T^n)} \\
        &\qquad\quad\leq \|\Delta (v_1-v_2) + H(x_1,\,\cdot\,,X_1 + D^2 v_1) - H(x_1,\,\cdot\,,X_1 + D^2 v_2)\|_{L^p(\T^n)} \\
        &\qquad\qquad+ \|H(x_1,\,\cdot\,,X_1 + D^2 v_2) - H(x_2,\,\cdot\,,X_2 + D^2 v_2)\|_{L^p(\T^n)}\\
        &\qquad\quad\leq \sqrt{1-\delta}\, \|D^2(v_1-v_2)\|_{L^p(\T^n)} \\ &\qquad\qquad+ \|\gamma A\|_\infty \lvert X_1 - X_2\rvert + \|\gamma\|_\infty \sup_{\alpha,\beta}\|f_{\alpha\beta}\|_{C^{0,\theta}(\overline{\Omega})} \lvert x_1 - x_2\rvert^\theta
    \end{align*}
    Recalling from Remark \ref{Rk: C1theta} and Lemma \ref{Lmm: cp Om leq Tn} that $c_{\T^n}(p)\geq c(p)> \sqrt{1-\delta}$, we can absorb the first term on the right-hand side into the left-hand side to obtain
    \begin{align*}
        \lvert v_1 - v_2\rvert_{W^{2,p}(\T^n)} \lesssim_{n,\|\gamma\|_\infty,p,\delta} ( \|u\|_{C^{2,\theta}(\overline{\Omega})} + \sup_{(\alpha,\beta)\in\mathcal{A}\times \mathcal{B}}\|f_{\alpha\beta}\|_{C^{0,\theta}(\overline{\Omega})}  )\lvert x_1 - x_2\rvert^\theta,
    \end{align*}
    which yields \eqref{v prep 2} in view of the Sobolev embeddings and $v_1-v_2\in L^p_0(\T^n)$. Finally, to show \eqref{dist bd}, we fix $\tilde x\in \partial\Omega$ such that $\lvert x-\tilde x\rvert = \mathrm{dist}(x,\partial\Omega)$ and note that
    \begin{align*}
        \lvert (u_\eps - u)(x) \rvert &= \lvert (u_\eps - u)(x) - (u_\eps - u)(\tilde x)\rvert\\ &\leq \|\nabla(u_\eps - u)\|_{L^\infty(\Omega)} \lvert x-\tilde x\rvert \lesssim_{n,p,\Omega} (\lvert u_\eps\rvert_{W^{2,p}(\Omega)} + \|u\|_{C^{2,\theta}(\overline{\Omega})}) \mathrm{dist}(x,\partial\Omega).
    \end{align*}
    The Sobolev embeddings together with the uniform bound \eqref{unf w2pbd} then yield \eqref{dist bd}.

\smallskip
\noindent    \textbf{Step 2.} We fix a constant $R > \max_{x\in \overline{\Omega}} (1+\lvert x\rvert^2)^{1/2}$ and a constant
    \begin{align}\label{RK}
        K > \max\left\{2c,\frac{\|\gamma\|_\infty \sup_{\alpha,\beta}\|f_{\alpha\beta}\|_{C^{\theta}(\overline{\Omega})} + \left(\sqrt{1-\delta} + \sqrt{n}\right)\|u\|_{C^{2,\theta}(\overline{\Omega})}}{n(1-\sqrt{1-\delta})}\right\},
    \end{align}
    and introduce $\Phi_\eps:\overline{\Omega}\rightarrow \R$ given by
    \begin{align*}
        \Phi_\eps(x) := u_\eps(x) - u(x) - \eps^2 v\left(\frac{x}{\eps};x,D^2 u(x)\right) - K\eps^\theta \left(R^2 - \lvert x\rvert^2\right).
    \end{align*}    
    Recalling \eqref{v prep 1}, \eqref{v prep 2}, and $u_\eps,u\in W^{2,p}(\Omega)$, we observe that $\Phi_\eps\in C^{0,\theta}(\overline{\Omega})$. We fix $x_0 \in \mathrm{argmax}_{ \overline{\Omega}} \Phi_\eps$, and aim to show that $\Phi_\eps(x_0)\leq 0$. We distinguish two cases.
    
    Case 1: Suppose $\mathrm{dist}(x_0,\partial\Omega)\leq \eps$. Then, by \eqref{dist bd}, \eqref{v prep 1}, and \eqref{RK} we have
    \begin{align*}
        \Phi_\eps(x_0) \leq c\, \mathrm{dist}(x_0,\partial\Omega) + c\,\eps^2 - K\eps^\theta \leq (2c-K)\eps^\theta \leq 0\quad\text{for }\eps\leq 1.
    \end{align*}
    Case 2: Suppose $\mathrm{dist}(x_0,\partial\Omega) > \eps$. In this case $\overline{B_\eps(x_0)}\subset \Omega$ and we introduce the perturbed test function $\Psi_\eps\in W^{2,p}(B_\eps(x_0))$ given by
    \begin{align}\label{Psiepsrate}
        \Psi_\eps(x) := u(x) + \eps^2 v\left(\frac{x}{\eps};x_0,D^2 u(x_0)\right) + K\eps^\theta \left(R^2 - \lvert x\rvert^2 + \frac{1}{2}\lvert x-x_0\rvert^2\right).
    \end{align}
    Using \eqref{HLip}, \eqref{cell}, and \eqref{hom prob}, we obtain for a.e. $x\in B_\eps(x_0)$ that
    \begin{align*}
    &H\left(x,\frac{x}{\eps},D^2 \Psi_\eps(x)\right) = H\left(x,\frac{x}{\eps},D^2u(x)
    + D^2v\left(\frac{x}{\eps};x_0,D^2u(x_0)\right)
    - K\eps^\theta I_n\right) \\
    &\geq H\left(x,\frac{x}{\eps},D^2u(x)
    + D^2v\left(\frac{x}{\eps};x_0,D^2u(x_0)\right)
    \right) + \mathrm{tr}(K\eps^\theta I_n)  -\sqrt{1-\delta}\,\lvert K\eps^\theta I_n\rvert \\
    &\geq H\left(x,\frac{x}{\eps},D^2u(x_0)
    + D^2v\left(\frac{x}{\eps};x_0,D^2u(x_0)\right)
    \right) \\ &\qquad- \left(\sqrt{1-\delta} + \sqrt{n}\right)\lvert D^2 u(x)-D^2 u(x_0)\rvert + n(1-\sqrt{1-\delta})K\eps^\theta \\
    &\geq \overline{H}\left(x_0,D^2u(x_0)\right) - \|\gamma\|_\infty \sup_{\alpha,\beta}\|f_{\alpha\beta}\|_{C^{\theta}(\overline{\Omega})} \lvert x-x_0\rvert^\theta
    \\ &\qquad- \left(\sqrt{1-\delta} + \sqrt{n}\right)\|u\|_{C^{2,\theta}(\overline{\Omega})} \lvert x-x_0\rvert^\theta + n(1-\sqrt{1-\delta})K\eps^\theta \\
    &\geq \left(n(1-\sqrt{1-\delta})K - \|\gamma\|_\infty \sup_{\alpha,\beta}\|f_{\alpha\beta}\|_{C^{\theta}(\overline{\Omega})} - \left(\sqrt{1-\delta} + \sqrt{n}\right)\|u\|_{C^{2,\theta}(\overline{\Omega})}  \right)\eps^\theta.
    \end{align*}
    Note that the right-hand side of the above inequality is positive by \eqref{RK}. Let $a_\eps\in L^\infty(B_\eps(x_0);\R^{n\times n}_\mathrm{sym})$ be defined by \eqref{linearization} with $\Psi_\eps$ given by \eqref{Psiepsrate}. Then, we have that $\lvert a_\eps - I_n\rvert \leq \sqrt{1-\delta}$, $a_\eps \geq (1-\sqrt{1-\delta})I_n$, and $a_\eps:D^2 (u_\eps - \Psi_\eps) = H(\cdot,\frac{\cdot}{\eps},D^2 \Psi_\eps) > 0$ almost everywhere in $B_\eps(x_0)$. Hence, by the maximum principle, 
    \begin{align}\label{for contr1}
    (u_\eps - \Psi_\eps)(x_0)\leq \max_{\partial B_\eps(x_0)} (u_\eps - \Psi_\eps).
    \end{align}
    Note that for any $x\in \partial B_\eps(x_0)$, we can use $x_0\in \mathrm{argmax}_{ \overline{\Omega}} \Phi_\eps$ and \eqref{v prep 2} to see that
    \begin{align*}
        &(u_\eps - \Psi_\eps)(x)-(u_\eps - \Psi_\eps)(x_0) \\
        &= \Phi_\eps(x) - \Phi_\eps(x_0) + \eps^2\left[v\left(\frac{x}{\eps};x,D^2u(x)\right) - v\left(\frac{x}{\eps};x_0,D^2u(x_0)\right)\right] - \frac{1}{2}K\eps^\theta \lvert x - x_0\rvert^2 \\
        &\leq c\eps^2\lvert x-x_0\rvert^\theta - \frac{1}{2}K\eps^\theta \lvert x - x_0\rvert^2 = \frac{1}{2}(2c-K) \eps^{2+\theta}<0,
    \end{align*}
    which contradicts \eqref{for contr1}. So, Case 2 cannot occur.

Altogether, we have shown that $\max_{\overline{\Omega}}\Phi_\eps = \Phi_\eps(x_0) \leq 0$, and hence, using \eqref{v prep 1},
\begin{align*}
    u_\eps(x) - u(x) &\leq \eps^2 v\left(\frac{x}{\eps};x,D^2 u(x)\right) + K\eps^\theta \left(R^2 - \lvert x\rvert^2\right) \leq (c+KR^2)\eps^\theta 
\end{align*}
for $\eps\leq 1$. A similar argument can be used to show the opposite bound of \eqref{rate}.
\end{proof}

\section{Numerical homogenization}\label{Sec: 4}

\subsection{{\boldmath$W^{2,p}$}-conforming finite element approximation of the corrector}\label{Sec: 4.1}

\subsubsection{Setting}\label{Sec: Setting FEM v}
Let the situation be as in Theorem \ref{Thm: cell}. Recall that $p>n$ and $q := \frac{p}{p-1}$ denotes the dual exponent. For $r\in \{p,q\}$, we write $W_r := W^{2,r}(\mathbb{T}^n)\cap L^r_0(\mathbb{T}^n)$ and $\|\cdot\|_{\Delta,r} := \|\Delta(\cdot)\|_{L^r(\mathbb{T}^n)}$, and note that $(W_r,\|\cdot\|_{\Delta,r})$ is a Banach space.

Let $(x,X)\in \overline{\Omega}\times \R^{n\times n}_{\mathrm{sym}}$ be fixed. In this section, we discuss the finite element approximation of the solution $v = v(\,\cdot\,;x,X)\in W_p$ to the cell problem \eqref{cell}, i.e.,
\begin{align}\label{cell ag}
    \text{seek}\quad v\in W_p\quad\text{s.t.}\quad H(x,\,\cdot\,,X+D^2 v) = \int_{\mathbb{T}^n} H(x,\,\cdot\,,X+D^2 v)\quad\text{a.e. in }\mathbb{T}^n.
\end{align}
Let us also recall the operators $\mathcal{C}_1,\mathcal{C}_2:W_p\rightarrow W_q^\ast$ given by
\begin{align*}
    \langle \mathcal{C}_1(w),z\rangle := \int_{\T^n} (H(x,\cdot,X+D^2 w)) (-\Delta z),\qquad
    \langle \mathcal{C}_2(w),z\rangle := \int_{\T^n} (-\Delta w) (-\Delta z) 
\end{align*}
for $w\in W_p$ and $z\in W_q$, which were central for the proof of Theorem \ref{Thm: cell}.

\subsubsection{Finite element approximation}
Our starting point for the construction is the observation that problem \eqref{cell ag} is equivalent to finding $v\in W_p$ such that $\mathcal{C}_1(v)= 0$. For chosen closed linear subspaces 
\begin{align*}
    W_{h,p}\subset W_p,\qquad W_{h,q}\subset W_q,
\end{align*}
we consider the following finite element method:
\begin{align}\label{PG}
    \text{seek}\quad v_h\in W_{h,p}\quad\text{s.t.}\quad \langle \mathcal{C}_1(v_h),z_h\rangle= 0\quad \forall z_h\in W_{h,q}.
\end{align}
The key technical tool in the error analysis of the resulting $W_p$-conforming finite element method is the fact that the nonlinear operator $\mathcal{C}_1$ is near the linear bijection $\mathcal{C}_2$, which we established in the proof of Theorem \ref{Thm: cell}. 

\begin{thrm}[finite element approximation of \eqref{cell ag}]\label{Thm: FEM for v}
    Let the setting be as described above. Suppose that there exists a Fortin operator $\Pi_h:W_q \rightarrow W_{h,q}$ such that
    \begin{align}\label{Fortin}
        \langle \mathcal{C}_2(w_h),z-\Pi_h (z)\rangle = 0,\quad \|\Pi_h (z)\|_{\Delta,q}\leq C_{\Pi} \|z\|_{\Delta,q} \;\; \forall (w_h,z)\in W_{h,p}\times  W_q
    \end{align}
    for some constant $0\leq C_\Pi < \xi(p)c_{\T^n}(p)(1-\delta)^{-1/2}$, and that the following compatibility condition holds for all $z_h\in W_{h,q}$:
    \begin{align}\label{compat}
        \left[\langle \mathcal{C}_2(w_h),z_h\rangle = 0\quad\forall w_h\in W_{h,p}\right]\;\Longrightarrow\; \left[\langle \mathcal{C}_1(w_h),z_h\rangle = 0\quad\forall w_h\in W_{h,p}\right].
    \end{align}
    Then, there exists a unique solution $v_h\in W_{h,p}$ to \eqref{PG}, and we have the error bound
    \begin{align*}
        \lvert v-v_h\rvert_{W^{2,p}(\T^n)} \leq \frac{1}{c_{\T^n}(p)} \left(1 + \frac{C_\Pi}{\xi(p)} \frac{\xi(p)c_{\T^n}(p) + \sqrt{1-\delta} }{\xi(p)c_{\T^n}(p) - C_\Pi\sqrt{1-\delta}}  \right)\inf_{w_h\in W_{h,p}} \|v-w_h\|_{\Delta,p},
    \end{align*}
    where $v\in W_p$ denotes the unique solution to \eqref{cell ag}.
\end{thrm}

\begin{proof}
We introduce the bounded linear operator
\begin{align*}
    \mathcal{C}_{h,2}:W_{h,p} \rightarrow Y_h,\qquad w_h\mapsto \mathcal{C}_{h,2}(w_h) :=  \left.\mathcal{C}_2 (w_h)\right\rvert_{W_{h,q}},
\end{align*}
where $Y_h := \{ \left.\mathcal{C}_2 (w_h)\right\rvert_{W_{h,q}} : w_h\in W_{h,p}\} \subset W_{h,q}^\ast$. Further, we introduce the operator
\begin{align*}
    \mathcal{C}_{h,1}:W_{h,p} \rightarrow Y_h,\qquad w_h\mapsto \mathcal{C}_{h,1}(w_h) :=  \left.\mathcal{C}_1 (w_h)\right\rvert_{W_{h,q}},
\end{align*}
which is yet to be shown to be well-defined. Then, the problem \eqref{PG} is equivalent to finding $v_h\in W_{h,p}$ such that $\mathcal{C}_{h,1}(v_h) = 0$. We split the proof into four steps.

\textbf{Step 1.} We show that $\mathcal{C}_{h,1}(W_{h,p})\subset Y_h$, and hence, $\mathcal{C}_{h,1}$ is well-defined.

First, we use Lemma \ref{Lmm: xi expl form}, bijectivity of $\left.\Delta\right\rvert_{W_q}:W_q\rightarrow L^q_0(\T^n)$, and the properties of $\Pi_h$ from \eqref{Fortin} to see that for any $w_h\in W_{h,p}$ we have
\begin{align}\label{long chain}
\begin{split}
    &\|w_h\|_{\Delta,p}     = \sup_{g\in L^q(\T^n)\backslash\{0\}} \frac{\langle \mathcal{C}_2(w_h),(\Delta_{W_q\rightarrow L^q_0(\T^n)})^{-1}(g-\int_{\T^n} g)\rangle}{\|g\|_{L^q(\T^n)}} \\
    &\quad\leq \sup_{z\in W_q\backslash\{0\}} \frac{\lvert \langle \mathcal{C}_2(w_h),z\rangle\rvert}{\|z\|_{\Delta,q}}\sup_{g\in L^q(\T^n)\backslash\{0\}} \frac{\|g-\int_{\T^n} g\|_{L^q(\T^n)}}{\|g\|_{L^q(\T^n)}} \\
    &\quad= (\xi(p))^{-1}\sup_{z\in W_q\backslash\{0\}} \frac{\lvert \langle \mathcal{C}_2(w_h),z\rangle\rvert}{\|z\|_{\Delta,q}} \\
    &\quad= (\xi(p))^{-1}\sup_{z\in W_q\backslash\{0\}} \frac{\lvert \langle \mathcal{C}_{h,2}(w_h),\Pi_h (z)\rangle\rvert}{\|z\|_{\Delta,q}}\leq \frac{C_\Pi}{\xi(p)}  \sup_{z_h\in W_{h,q}\backslash\{0\}} \frac{\lvert \langle \mathcal{C}_{h,2}(w_h),z_h  \rangle \rvert}{\|z_h\|_{\Delta,q}}.
\end{split}
\end{align}
In view of the above estimate, completeness of $W_{h,p}$ as a closed subspace of the Banach space $W_p$, and continuity of $\mathcal{C}_{h,2}$, we find that $Y_h$ is a closed linear subspace of $W_{h,q}^\ast$.  Hence, the inclusion map $j_h:Y_h \hookrightarrow W_{h,q}^\ast$ is bounded, linear, and has closed range. By the closed range theorem (see, e.g., \cite{OD96}), we have
\begin{align}\label{closed range impl}
    Y_h = j_h(Y_h) = \{y_h\in W_{h,q}^\ast\,|\, \langle \phi_h,y_h\rangle = 0 \;\;\forall \phi_h\in \mathrm{ker}(j_h^\ast)\},
\end{align}
where $j_h^\ast: W_{h,q}^{\ast\ast}\rightarrow Y_h^\ast$, $\psi_h \mapsto \psi_h\circ j_h$. Let $w_h\in W_{h,p}$ and set $y_h :=\mathcal{C}_{h,1}(w_h)$. Note that $y_h = \left.\mathcal{C}_1 (w_h)\right\rvert_{W_{h,q}}\in W_{h,q}^\ast$. Now let $\phi_h\in \mathrm{ker}(j_h^\ast)\subset W_{h,q}^{\ast\ast}$ and observe that, by reflexivity of $W_{h,q}$ as a closed linear subspace of the reflexive Banach space $W_q$, there exists $z_h\in W_{h,q}$ such that $\langle \phi_h,\cdot\rangle_{W_{h,q}^{\ast\ast},W_{h,q}^\ast} = \langle \cdot ,z_h\rangle_{W_{h,q}^\ast,W_{h,q}}$. Then, we have that $\langle \mathcal{C}_2(x_h),z_h\rangle =  \langle \phi_h, j_h( \mathcal{C}_{h,2} (x_h))\rangle = \langle j_h^\ast(\phi_h),\mathcal{C}_{h,2} (x_h)\rangle = 0$ for all $x_h\in W_{h,p}$, and hence, $\langle \phi_h,y_h\rangle = \langle y_h,z_h\rangle = \langle \mathcal{C}_1(w_h),z_h\rangle = 0$ by assumption \eqref{compat}. In view of \eqref{closed range impl}, we obtain $y_h\in Y_h$. Altogether, there holds $\mathcal{C}_{h,1}(W_{h,p})\subset Y_h$.

\textbf{Step 2.} We show that $\mathcal{C}_{h,1}$ is near $\mathcal{C}_{h,2}$.

Since $Y_h$ is a closed linear subspace of the Banach space $W_{h,q}^\ast$, we know that $Y_h$ equipped with the inherited norm
\begin{align*}
    \| y_h \|_{Y_h} := \sup_{z_h\in W_{h,q}\backslash\{0\}} \frac{\lvert \langle y_h,z_h  \rangle_{W_{h,q}^\ast,W_{h,q}} \rvert}{\|z_h\|_{\Delta,q}}\qquad\text{for }\; y_h\in Y_h
\end{align*}
is a Banach space. Using linearity of $\mathcal{C}_{h,2}$ and $\mathcal{C}_2$, and the estimate \eqref{K tilde near} that established that $\mathcal{C}_1$ is near $\mathcal{C}_2$, we have for any $w_h,w_h' \in W_{h,p}$ that
\begin{align}\label{nearste2}
\begin{split}
    &\| \mathcal{C}_{h,2}(w_h) - \mathcal{C}_{h,2}(w_h') - [\mathcal{C}_{h,1}(w_h) - \mathcal{C}_{h,1}(w_h')] \|_{Y_h} \\
    &\quad =\left\| \left.\left(\mathcal{C}_{2}(w_h-w_h') - [\mathcal{C}_{1}(w_h) - \mathcal{C}_{1}(w_h')]\right)\right\rvert_{W_{h,q}}\right\|_{Y_h} \\
    &\quad\leq  \| \mathcal{C}_{2}(w_h-w_h') - [\mathcal{C}_{1}(w_h) - \mathcal{C}_{1}(w_h')]\|_{W_q^\ast} \\
    &\quad\leq (\xi(p)c_{\T^n}(p))^{-1}\sqrt{1-\delta}\, \|\mathcal{C}_2(w_h-w_h')\|_{W_q^\ast} \\
    &\quad\leq C_\Pi(\xi(p)c_{\T^n}(p))^{-1}\sqrt{1-\delta}\, \| \mathcal{C}_{h,2}(w_h-w_h')\|_{Y_h}.
\end{split}
\end{align}
Noting that $C_\Pi(\xi(p)c_{\T^n}(p))^{-1}\sqrt{1-\delta} \in [0,1)$, we conclude that $\mathcal{C}_{h,1}$ is near $\mathcal{C}_{h,2}$.

\textbf{Step 3.} We observe that $\mathcal{C}_{h,2}:W_{h,p}\rightarrow Y_h$ is a bijection. Indeed, $\mathcal{C}_{h,2}$ is surjective by definition of $Y_h$, and injective by \eqref{long chain} and linearity of $\mathcal{C}_{h,2}$.  

\textbf{Step 4.} We can now conclude the proof. In view of the previous steps and Lemma \ref{Lmm: near}, we find that $\mathcal{C}_{h,1}:W_{h,p}\rightarrow Y_h$ is a bijection and we have the bound
\begin{align}\label{vh bd prep}
    \| \mathcal{C}_{h,2}(w_h - w_h') \|_{Y_h} \leq \frac{\| \mathcal{C}_{h,1}(w_h) - \mathcal{C}_{h,1}(w_h')\|_{Y_h}}{1 - C_\Pi(\xi(p)c_{\T^n}(p))^{-1}\sqrt{1-\delta}}\qquad \forall w_h,w_h'\in W_{h,p}.
\end{align}
In particular, there exists a unique $v_h\in W_{h,p}$ such that $\mathcal{C}_{h,1}(v_h) = 0$, or equivalently, there exists a unique solution $v_h\in W_{h,p}$ to \eqref{PG}. Further, using \eqref{long chain}, \eqref{vh bd prep}, $\mathcal{C}_{h,1}(v_h) = 0$, and $\mathcal{C}_1(v) = 0$, we have for any $w_h\in  W_{h,p}$ that
\begin{align}\label{err bd comb1}
\begin{split}
    \|v-v_h\|_{\Delta,p} &\leq \|v - w_h\|_{\Delta,p} + \|w_h-v_h\|_{\Delta,p}\\
    &\leq \|v - w_h\|_{\Delta,p} + C_\Pi (\xi(p))^{-1} \|\mathcal{C}_{h,2}(w_h-v_h)\|_{Y_h} \\
    &\leq \|v - w_h\|_{\Delta,p} + \frac{C_\Pi (\xi(p))^{-1}}{1 - C_\Pi(\xi(p)c_{\T^n}(p))^{-1}\sqrt{1-\delta}}\,\| \mathcal{C}_{h,1}(w_h)\|_{Y_h} \\
    &\leq \|v - w_h\|_{\Delta,p} + \frac{c_{\T^n}(p)C_\Pi }{\xi(p)c_{\T^n}(p) - C_\Pi\sqrt{1-\delta}}\,\| \mathcal{C}_{1}(v) - \mathcal{C}_{1}(w_h)\|_{W_q^\ast}.
\end{split}
\end{align}
For the final term on the right-hand side, we can use \eqref{K tilde near} to obtain
\begin{align}\label{err bd comb2}
\begin{split}
    \| \mathcal{C}_{1}(v) - \mathcal{C}_{1}(w_h)\|_{W_q^\ast} &\leq \| \mathcal{C}_{2}(v-w_h)\|_{W_q^\ast} + \| \mathcal{C}_{1}(v) - \mathcal{C}_{1}(w_h) - \mathcal{C}_{2}(v-w_h)\|_{W_q^\ast} \\
    &\leq \left( 1 +  (\xi(p)c_{\T^n}(p))^{-1}\sqrt{1-\delta}\right)\| \mathcal{C}_{2}(v-w_h)\|_{W_q^\ast} \\
    &\leq \left( 1 +  (\xi(p)c_{\T^n}(p))^{-1}\sqrt{1-\delta}\right)\| v-w_h\|_{\Delta,p}
\end{split}
\end{align}
for any $w_h\in W_{h,p}$. Combining \eqref{err bd comb1} and \eqref{err bd comb2}, we conclude that
\begin{align*}
    &c_{\T^n}(p)\lvert v-v_h\rvert_{W^{2,p}(\T^n)} \leq \|v-v_h\|_{\Delta,p} \\ &\qquad\qquad\qquad\leq \| v-w_h\|_{\Delta,p} + \frac{c_{\T^n}(p)C_\Pi \left( 1 +  (\xi(p)c_{\T^n}(p))^{-1}\sqrt{1-\delta}\right) }{\xi(p)c_{\T^n}(p) - C_\Pi\sqrt{1-\delta}}  \| v-w_h\|_{\Delta,p}
\end{align*}
for any $w_h\in W_{h,p}$, which yields the claimed error bound.
\end{proof}

Let us make two important remarks.

\begin{rmrk}\label{Rk: compat}
    Suppose that $W_{h,p} = W_{h,q} = S_h$ for some closed linear subspace $S_h \subset W_p\cap W_q$. Then, the compatibility condition \eqref{compat} is satisfied. Indeed, let $z_h\in S_h \subset W_2$ be such that $\langle \mathcal{C}_2(w_h),z_h\rangle = 0$ for all $w_h\in S_h$. Then, we have $\|z_h\|_{\Delta,2}^2 = \langle \mathcal{C}_2(z_h),z_h\rangle = 0$ and hence $z_h = 0$ since $\|\cdot\|_{\Delta,2}$ is a norm on $W_2$. Therefore, $\langle \mathcal{C}_1(w_h),z_h\rangle = \langle \mathcal{C}_1(w_h),0\rangle = 0$ for all $w_h\in S_h$.
\end{rmrk}

\begin{rmrk}\label{Rk: Fortin p = 2}
    Suppose that $p = 2$ and that $W_{h,p} = W_{h,q} = S_h$ for some closed linear subspace $S_h \subset W_2$. Then, both \eqref{Fortin} and \eqref{compat} are satisfied. Indeed, since $W_2$ is a Hilbert space with inner product $(z_1,z_2)_{\Delta,2} := (\Delta z_1,\Delta z_2)_{L^2(\T^n)}$ for $z_1,z_2\in W_2$, there exists a unique operator $\Pi_h: W_2 \rightarrow S_h$ such that $(\Pi_h (z), s_h)_{\Delta,2} = (z,s_h)_{\Delta,2}$ $\forall s_h\in S_h$ for any $z\in W_2$, and \eqref{Fortin} holds with $C_\Pi = 1 < (1-\delta)^{-1/2} =  \xi(2)c_{\T^n}(2)(1-\delta)^{-1/2}$. The compatibility condition \eqref{compat} is satisfied by Remark \ref{Rk: compat}.
\end{rmrk}

\subsection{Mixed finite element approximation of the corrector}\label{Sec: 4.2}

Let the setting be as in Section \ref{Sec: Setting FEM v}. Motivated by the fact that $\overline{H}(x,X)$ introduced in Definition \ref{Def: Hbar} only depends on $D^2 v$ rather than the function $v = v(\cdot;x,X)$ itself, we now present and analyze a mixed finite element method for the approximation of the corrector $v$. Another benefit of a mixed approach is the avoidance of $W^{2,p}$-conforming finite elements in a practical implementation.

\subsubsection{Mixed formulation}

For $r\in \{p,q\}$, we write $\tilde W_r := W^{1,r}(\T^n;\R^n)\cap L^r_0(\T^n;\R^n)$ and $\|\cdot\|_{D,r}:= \|D(\cdot)\|_{L^r(\T^n)} = \lvert\, \cdot\,\rvert_{W^{1,r}(\T^n)}$, and note that $(\tilde W_r,\|\cdot\|_{D,r})$ is a Banach space. We begin by introducing the operators
\begin{align*}
    \mathcal{\tilde C}_i:\tilde W_p\rightarrow \tilde W_q^\ast&,\qquad w\mapsto \mathcal{\tilde C}_i(w)
\end{align*}
for $i\in\{1,2\}$, where the actions of $\mathcal{\tilde C}_1(w)$ and $\mathcal{\tilde C}_2(w)$ are given by
\begin{align*}
    \langle \mathcal{\tilde C}_1(w),z\rangle &:= \int_{\T^n} (H(x,\cdot,X+Dw)) (-\nabla\cdot z) + \frac{1}{2}\int_{\T^n} (Dw-Dw^{\mathrm{T}}):(Dz-Dz^{\mathrm{T}}) ,\\
    \langle \mathcal{\tilde C}_2(w),z\rangle &:= \int_{\T^n} Dw:Dz 
\end{align*}
for $z\in \tilde W_q$. Note that for any $w\in W_p$ and $z\in W_q$ we have $\nabla w\in \tilde W_p$, $\nabla z\in \tilde W_q$, and $\langle \mathcal{C}_1(w),z\rangle = \langle \mathcal{\tilde C}_1(\nabla w),\nabla z\rangle$. The stabilization involved in the definition of $\mathcal{\tilde C}_1$ is a curl-curl stabilization in the spirit of \cite{Gal17b,Gal19}, extended to arbitrary dimension.

Our starting point for the construction of the finite element method is the following variational formulation for the gradient of the corrector:
\begin{align}\label{rho prob}
    \text{seek}\quad \rho\in \tilde W_p\quad\text{s.t.}\quad \mathcal{\tilde C}_1(\rho)= 0.
\end{align}
For its analysis, we introduce the constants
\begin{align}\label{ctilp dq}
\begin{split}
        \tilde c_{\T^n}(p) &:= \inf_{w\in \tilde W_p\backslash\{0\}} \frac{\|\Delta_{\tilde W_p \rightarrow \tilde W_q^\ast} (w)\|_{\tilde W_q^\ast}}{\|w\|_{D,p}} = \|\tilde C_2^{-1}\|^{-1},\\
        \tilde d_{\T^n}(q) &:= \inf_{z\in \tilde W_q, \nabla\cdot z\neq 0} \frac{\|z\|_{D,q}}{\|\nabla\cdot z\|_{L^q(\T^n)}} = \|\div \|_{(\tilde W_q,\|\cdot\|_{D,q})\rightarrow (L^q_0(\T^n),\|\cdot\|_{L^q(\T^n)})}^{-1},
\end{split}
\end{align}
and note that $\tilde c_{\T^n}(p) > 0$ and $\tilde d_{\T^n}(q)\in [ n^{-\frac{1}{2}},1]$, where well-definedness of the former follows from the fact that $\mathcal{\tilde C}_2 = -\Delta:\tilde W_p \rightarrow \tilde W_q^\ast$ is a bounded bijection.

\begin{thrm}[characterization of the gradient of the corrector]\label{thm: char rho}
    Let the setting be as described above, and suppose that $\delta > 1 - (\tilde c_{\T^n}(p)\tilde d_{\T^n}(q))^2$. Then, the problem \eqref{rho prob} has a unique solution $\rho\in \tilde W_p$, and there holds $\rho = \nabla v$, where $v\in W_p$ is the unique solution to \eqref{cell ag}.
\end{thrm}

\begin{proof}
    First, recall that $\mathcal{\tilde C}_2 = -\Delta:\tilde W_p \rightarrow \tilde W_q^\ast$ is a bounded bijection. Integration by parts reveals the Gaffney-type identity (see also \cite{LSW26})
    \begin{align}\label{Gaffney}
        \langle \mathcal{\tilde C}_2(w),z\rangle = \int_{\T^n} (\nabla\cdot w)(\nabla\cdot z) + \frac{1}{2}\int_{\T^n} (Dw-Dw^{\mathrm{T}}):(Dz-Dz^{\mathrm{T}})
    \end{align}
    for any $w\in \tilde W_p$ and $z\in \tilde W_q$.

    \textbf{Step 1.} We show that $\mathcal{\tilde C}_1$ is near $\mathcal{\tilde C}_2$. To this end, let $w_1,w_2\in \tilde W_p$. In view of \eqref{Gaffney}, we have for any $z\in \tilde W_q$ that
    \begin{align*}
        &\lvert \langle\mathcal{\tilde C}_1(w_1) - \mathcal{\tilde C}_1(w_2) -\mathcal{\tilde C}_2(w_1-w_2),z\rangle\rvert \\
        &=\int_{\T^n} (H(x,\cdot,X+Dw_1)-H(x,\cdot,X+Dw_2) + \nabla\cdot (w_1-w_2)) (-\nabla\cdot z). 
    \end{align*}
    Hence, using \eqref{ctilp dq} and \eqref{HLip}, we see that
    \begin{align}\label{tildene}
    \begin{split}
        &\| \mathcal{\tilde C}_2(w_1) - \mathcal{\tilde C}_2(w_2) - [\mathcal{\tilde C}_1(w_1) - \mathcal{\tilde C}_1(w_2)] \|_{\tilde W_q^\ast} \\&\qquad= \sup_{z\in \tilde W_q\backslash\{0\}} \|z\|_{D,q}^{-1}\;\lvert \langle\mathcal{\tilde C}_1(w_1) - \mathcal{\tilde C}_1(w_2) -\mathcal{\tilde C}_2(w_1-w_2),z\rangle\rvert  \\
    &\qquad\leq (\tilde d_{\T^n}(q))^{-1} \|H(x,\cdot,X+Dw_1)-H(x,\cdot,X+Dw_2) + \mathrm{tr}(D w_1-Dw_2)\|_{L^p(\T^n)} \\
    &\qquad \leq (\tilde d_{\T^n}(q))^{-1}\sqrt{1-\delta}\, \|D(w_1-w_2)\|_{L^p(\T^n)} \\
    &\qquad \leq (\tilde c_{\T^n}(p)\tilde d_{\T^n}(q))^{-1}\sqrt{1-\delta} \| \mathcal{\tilde C}_2(w_1) - \mathcal{\tilde C}_2(w_2)\|_{\tilde W_q^\ast}.
    \end{split}
    \end{align}
    Noting that $(\tilde c_{\T^n}(p)\tilde d_{\T^n}(q))^{-1}\sqrt{1-\delta} < 1$ since $\delta > 1 - (\tilde c_{\T^n}(p)\tilde d_{\T^n}(q))^2$, we deduce that $\mathcal{\tilde C}_1$ is near $\mathcal{\tilde C}_2$.

    \textbf{Step 2.} In view of Lemma \ref{Lmm: near} and the nearness established in Step 1, we conclude that $\mathcal{\tilde C}_1$ is bijective. In particular, there exists a unique $\rho\in \tilde W_p$ such that $\tilde C_1(\rho) = 0$, i.e., a unique solution to \eqref{rho prob}. Then, we have
    \begin{align*}
        \int_{\T^n} (H(x,\cdot,X+D\rho)) (-\Delta z) = \langle \mathcal{\tilde C}_1(\rho), \nabla z\rangle = 0\qquad \forall z\in W_q,
    \end{align*}
    and hence, $H(x,\cdot,X+D\rho) = \int_{\T^n} H(x,\cdot,X+D\rho)$ a.e. in $\T^n$. Therefore,
    \begin{align*}
        \frac{1}{2}\int_{\T^n} (D\rho-D\rho^{\mathrm{T}}):(Dz-Dz^{\mathrm{T}}) = \langle \mathcal{\tilde C}_1(\rho),z\rangle = 0\quad \forall z\in \tilde W_q,
    \end{align*}
    i.e., $\rho = \nabla \tilde{\rho}$ for some $\tilde \rho\in W_p$. Noting that $H(x,\cdot,X+D^2 \tilde \rho) = \int_{\T^n} H(x,\cdot,X+D^2 \tilde \rho)$ a.e. in $\T^n$, we use uniqueness of solutions to \eqref{cell ag} in $W_p$ to conclude that $\tilde \rho = v$.
\end{proof}

\subsubsection{Finite element approximation}

Let the setting be as in Theorem \ref{thm: char rho}. For chosen closed linear subspaces $\tilde W_{h,p}\subset \tilde W_p$ and $\tilde W_{h,q}\subset \tilde W_q$, we consider the following finite element method:
\begin{align}\label{PG mixed}
    \text{seek}\quad \rho_h\in \tilde W_{h,p}\quad\text{s.t.}\quad \langle \mathcal{\tilde C}_1(\rho_h),z_h\rangle= 0\quad \forall z_h\in \tilde W_{h,q}.
\end{align}
\begin{thrm}[finite element approximation of the gradient of the corrector]\label{Thm: FEM for v mixed}
    Let the setting be as described above. Suppose that there exists a Fortin operator $\tilde \Pi_h:\tilde W_q \rightarrow \tilde W_{h,q}$ and a constant $0\leq C_{\tilde \Pi} < \tilde c_{\T^n}(p)\tilde d_{\T^n}(q)(1-\delta)^{-1/2}$ such that
    \begin{align}\label{Fortin mixed}
        \langle \mathcal{\tilde C}_2(w_h),z-\tilde \Pi_h (z)\rangle = 0,\quad \lvert \tilde \Pi_h (z)\rvert_{W^{1,q}(\T^n)}\leq C_{\tilde \Pi} \lvert z\rvert_{W^{1,q}(\T^n)} \;\; \forall (w_h,z)\in \tilde W_{h,p}\times  \tilde W_q,
    \end{align}
    and that the following compatibility condition holds for all $z_h\in \tilde W_{h,q}$:
    \begin{align}\label{compat mixed}
        \left[\langle \mathcal{\tilde C}_2(w_h),z_h\rangle = 0\quad\forall w_h\in \tilde W_{h,p}\right]\;\Longrightarrow\; \left[\langle \mathcal{\tilde C}_1(w_h),z_h\rangle = 0\quad\forall w_h\in \tilde W_{h,p}\right].
    \end{align}
    Then, there exists a unique solution $\rho_h\in \tilde W_{h,p}$ to \eqref{PG mixed}, and we have the bound
    \begin{align*}
        \lvert \nabla v-\rho_h\rvert_{W^{1,p}(\T^n)} \leq \left(\hspace{-0.07cm} 1 \hspace{-0.05cm}+\hspace{-0.05cm} \frac{C_{\tilde \Pi}}{\tilde c_{\T^n}(p)} \frac{ \tilde c_{\T^n}(p)\tilde d_{\T^n}(q) +  \sqrt{1-\delta} }{\tilde c_{\T^n}(p)\tilde d_{\T^n}(q) - C_{\tilde \Pi}\sqrt{1-\delta}}   \right)\inf_{w_h\in \tilde W_{h,p}} \lvert \nabla v-w_h\rvert_{W^{1,p}(\T^n)}
    \end{align*}
    where $v\in W_p$ denotes the unique solution to \eqref{cell ag}.
\end{thrm}

\begin{proof}
We introduce the bounded linear operator
\begin{align*}
    \mathcal{\tilde C}_{h,2}: \tilde W_{h,p} \rightarrow \tilde Y_h,\qquad w_h\mapsto \mathcal{\tilde C}_{h,2}(w_h) :=  \mathcal{\tilde C}_2 (w_h)|_{\tilde W_{h,q}},
\end{align*}
where $\tilde Y_h := \{ \mathcal{\tilde C}_2 (w_h)|_{\tilde W_{h,q}} : w_h\in \tilde W_{h,p}\} \subset \tilde W_{h,q}^\ast$. Further, we introduce the operator
\begin{align*}
    \mathcal{\tilde C}_{h,1}:\tilde W_{h,p} \rightarrow \tilde Y_h,\qquad w_h\mapsto \mathcal{\tilde C}_{h,1}(w_h) :=  \mathcal{\tilde C}_1 (w_h)|_{\tilde W_{h,q}},
\end{align*}
which is yet to be shown to be well-defined. Then, the problem \eqref{PG mixed} is equivalent to finding $\rho_h\in \tilde W_{h,p}$ such that $\mathcal{\tilde C}_{h,1}(\rho_h) = 0$. The proof now proceeds similarly to the proof of Theorem \ref{Thm: FEM for v} and we focus on the required modifications.

\textbf{Step 1.} We show that $\mathcal{\tilde C}_{h,1}(\tilde W_{h,p})\subset \tilde Y_h$, and hence, $\mathcal{\tilde C}_{h,1}$ is well-defined.

The abstract functional analytic argument from Step 1 of the proof of Theorem \ref{Thm: FEM for v} goes through with \eqref{long chain} replaced by the bound
\begin{align}\label{long chain 2}
\begin{split}
    \|w_h\|_{D,p}     &\leq (\tilde c_{\T^n}(p))^{-1}\sup_{z\in \tilde W_q\backslash\{0\}} \frac{\lvert \langle \mathcal{\tilde C}_2(w_h),z\rangle\rvert}{\|z\|_{D,q}} \\
    &= (\tilde c_{\T^n}(p))^{-1}\sup_{z\in \tilde W_q\backslash\{0\}} \frac{\lvert \langle \mathcal{\tilde C}_{h,2}(w_h),\tilde \Pi_h (z)\rangle\rvert}{\|z\|_{D,q}}\leq \frac{C_{\tilde\Pi}}{\tilde c_{\T^n}(p)}  \sup_{z_h\in \tilde W_{h,q}\backslash\{0\}} \frac{\lvert \langle \mathcal{\tilde C}_{h,2}(w_h),z_h  \rangle \rvert}{\|z_h\|_{D,q}}.
\end{split}
\end{align}
for any $w_h\in \tilde W_{h,p}$. 

\textbf{Step 2.} We show that $\mathcal{\tilde C}_{h,1}$ is near $\mathcal{\tilde C}_{h,2}$.

Since $\tilde Y_h$ is a closed linear subspace of the Banach space $\tilde W_{h,q}^\ast$, we know that $\tilde Y_h$ equipped with the inherited norm
\begin{align*}
    \| y_h \|_{\tilde Y_h} := \sup_{z_h\in \tilde W_{h,q}\backslash\{0\}} \|z_h\|_{D,q}^{-1}\lvert \langle y_h,z_h  \rangle_{\tilde W_{h,q}^\ast,\tilde W_{h,q}} \rvert\qquad\text{for }\; y_h\in \tilde Y_h
\end{align*}
is a Banach space. Arguing as in \eqref{nearste2} with \eqref{tildene} taking the role of \eqref{K tilde near}, we find
\begin{align*}
    &\| \mathcal{\tilde C}_{h,2}(w_h) - \mathcal{\tilde C}_{h,2}(w_h') - [\mathcal{\tilde C}_{h,1}(w_h) - \mathcal{\tilde C}_{h,1}(w_h')] \|_{\tilde Y_h} \\
    &\quad\leq C_{\tilde \Pi}(\tilde c_{\T^n}(p)\tilde d_{\T^n}(q))^{-1}\sqrt{1-\delta}\, \| \mathcal{\tilde C}_{h,2}(w_h-w_h')\|_{\tilde Y_h}\qquad \forall w_h,w_h'\in \tilde W_{h,p},
\end{align*}
which implies that $\mathcal{\tilde C}_{h,1}$ is near $\mathcal{\tilde C}_{h,2}$ since $C_{\tilde \Pi}(\tilde c_{\T^n}(p)\tilde d_{\T^n}(q))^{-1}\sqrt{1-\delta} \in [0,1)$.

\textbf{Step 3.} We observe that $\mathcal{\tilde C}_{h,2}:\tilde W_{h,p}\rightarrow \tilde Y_h$ is a bijection. Indeed, $\mathcal{\tilde C}_{h,2}$ is surjective by definition of $\tilde Y_h$, and injective by \eqref{long chain 2} and linearity of $\mathcal{\tilde C}_{h,2}$.

\textbf{Step 4.} We can now conclude the proof. In view of the previous steps and Lemma \ref{Lmm: near}, we find that $\mathcal{\tilde C}_{h,1}:\tilde W_{h,p}\rightarrow \tilde Y_h$ is a bijection and we have the bound
\begin{align}\label{vh bd prep 2}
    \| \mathcal{\tilde C}_{h,2}(w_h - w_h') \|_{\tilde Y_h} \leq \frac{\| \mathcal{\tilde C}_{h,1}(w_h) - \mathcal{\tilde C}_{h,1}(w_h')\|_{\tilde Y_h}}{1 - C_{\tilde \Pi}(\tilde c_{\T^n}(p)\tilde d_{\T^n}(q))^{-1}\sqrt{1-\delta}}\qquad \forall w_h,w_h'\in \tilde W_{h,p}.
\end{align}
In particular, there exists a unique $\rho_h\in \tilde W_{h,p}$ such that $\mathcal{\tilde C}_{h,1}(\rho_h) = 0$, or equivalently, there exists a unique solution $\rho_h\in \tilde W_{h,p}$ to \eqref{PG mixed}. Further, using \eqref{long chain 2}, \eqref{vh bd prep 2}, $\mathcal{\tilde C}_{h,1}(\rho_h) = 0$, and $\mathcal{\tilde C}_1(\nabla v) = 0$, we argue analogously to \eqref{err bd comb1} and \eqref{err bd comb2} to obtain
\begin{align*}
    &\|\nabla v-\rho_h\|_{D,p} \leq \|\nabla v - w_h\|_{D,p} + \frac{\tilde d_{\T^n}(q)C_{\tilde \Pi} }{\tilde c_{\T^n}(p)\tilde d_{\T^n}(q) - C_{\tilde \Pi}\sqrt{1-\delta}}\,\| \mathcal{\tilde C}_{1}(\nabla v) - \mathcal{\tilde C}_{1}(w_h)\|_{\tilde W_q^\ast},\\
    &\| \mathcal{\tilde C}_{1}(\nabla v) - \mathcal{\tilde C}_{1}(w_h)\|_{\tilde W_q^\ast} \leq\left( 1 +  (\tilde c_{\T^n}(p)\tilde d_{\T^n}(q))^{-1}\sqrt{1-\delta}\right)\| \nabla v-w_h\|_{D,p}
\end{align*}
for any $w_h\in \tilde W_{h,p}$, which in combination yields
\begin{align*}
    \|\nabla v-\rho_h\|_{D,p} \leq \left( 1 + \frac{\tilde d_{\T^n}(q)C_{\tilde \Pi} \left( 1 +  (\tilde c_{\T^n}(p)\tilde d_{\T^n}(q))^{-1}\sqrt{1-\delta}\right) }{\tilde c_{\T^n}(p)\tilde d_{\T^n}(q) - C_{\tilde \Pi}\sqrt{1-\delta}}   \right)\| \nabla v-w_h\|_{D,p}
\end{align*}
for any $w_h\in \tilde W_{h,p}$, which yields the claimed error bound.
\end{proof}

Let us make two important remarks in the spirit of Remarks \ref{Rk: compat} and \ref{Rk: Fortin p = 2}.

\begin{rmrk}\label{Rk: compat 2}
    Suppose that $\tilde W_{h,p} = \tilde W_{h,q} = \tilde S_h$ for some closed linear subspace $\tilde S_h \subset \tilde W_p\cap \tilde W_q$. Then, the compatibility condition \eqref{compat mixed} is satisfied. Indeed, let $z_h\in \tilde S_h \subset \tilde W_2$ be such that $\langle \mathcal{\tilde C}_2(w_h),z_h\rangle = 0$ for all $w_h\in \tilde S_h$. Then, we have $\|z_h\|_{D,2}^2 = \langle \mathcal{\tilde C}_2(z_h),z_h\rangle = 0$ and hence $z_h = 0$ since $\|\cdot\|_{D,2}$ is a norm on $\tilde W_2$. Therefore, $\langle \mathcal{\tilde C}_1(w_h),z_h\rangle = \langle \mathcal{\tilde C}_1(w_h),0\rangle = 0$ for all $w_h\in \tilde S_h$.
\end{rmrk}

\begin{rmrk}\label{Rk: Fortin p = 2 2}
    Suppose that $p = 2$ and that $\tilde W_{h,p} = \tilde W_{h,q} = \tilde S_h$ for some closed linear subspace $\tilde S_h \subset \tilde W_2$. Then, both \eqref{Fortin mixed} and \eqref{compat mixed} are satisfied. Indeed, since $\tilde W_2$ is a Hilbert space with inner product $(z_1,z_2)_{D,2} := (D z_1,D z_2)_{L^2(\T^n)}$ for $z_1,z_2\in \tilde W_2$, there exists a unique operator $\tilde \Pi_h: \tilde W_2 \rightarrow \tilde S_h$ such that $(\tilde \Pi_h (z), s_h)_{D,2} = (z,s_h)_{D,2}$ $\forall s_h\in \tilde S_h$
    for any $z\in \tilde W_2$, and \eqref{Fortin mixed} holds with $C_{\tilde\Pi} = 1 < (1-\delta)^{-1/2} =  \tilde c_{\T^n}(2)\tilde{d}_{\T^n}(2)(1-\delta)^{-1/2}$. The compatibility condition \eqref{compat mixed} is satisfied by Remark \ref{Rk: compat 2}.
\end{rmrk}

\subsection{Numerical approximation of the effective Hamiltonian}\label{Sec: 4.3}

For the approximation of the effective Hamiltonian, we take the simplest method, that is, we choose $p = 2$ and $\tilde W_{h,p} = \tilde W_{h,q} = \tilde S_h$ for some family of closed linear subspaces $(\tilde S_h)_{h>0} \subset \tilde W_2 = H^1(\T^n)\cap L^2_0(\T^n)$ with the approximation property that $\inf_{w_h\in \tilde S_h} \lvert g-w_h\rvert_{H^1(\T^n)} \rightarrow 0$ as $h\rightarrow 0$ for all $g\in \tilde W_2$. 

Recall from Remark \ref{Rk: Fortin p = 2 2} that all assumptions of Theorem \ref{Thm: FEM for v mixed} are met. For a pair $(x,X)\in \overline{\Omega}\times \R^{n\times n}_{\mathrm{sym}}$, let $\rho_h(\,\cdot\,;x,X)$ denote the finite element approximation to $\nabla v(\,\cdot\,;x,X)$ from $\tilde S_h$ obtained from Theorem \ref{Thm: FEM for v mixed}. In view of Definition \ref{Def: Hbar}, we introduce the approximate effective Hamiltonian
\begin{align*}
    \overline{H}_h:\overline{\Omega}\times \R^{n\times n}_{\mathrm{sym}}\rightarrow \R,\qquad  (x,X)\mapsto\overline{H}_h(x,X) := \int_{\mathbb{T}^n} H(x,\,\cdot\,,X+D\rho_h(\,\cdot\,;x,X)).
\end{align*}

\begin{thrm}[Numerical approximation of the effective Hamiltonian]
    Let the situation be as described above. Then, for any $(x,X)\in \overline{\Omega}\times \R^{n\times n}_{\mathrm{sym}}$ we have the bound
    \begin{align*}
        \lvert \overline{H}(x,X) - \overline{H}_h(x,X)\rvert \leq 2\frac{\sqrt{1-\delta} + \sqrt{n}}{1-\sqrt{1-\delta}}  \inf_{w_h\in \tilde S_h} \lvert \nabla v-w_h\rvert_{H^1(\T^n)},
    \end{align*}
    where $\overline{H}$ denotes the effective Hamiltonian defined in Definition \ref{Def: Hbar}.
\end{thrm}

\begin{proof}
Let $(x,X)\in \overline{\Omega}\times \R^{n\times n}_{\mathrm{sym}}$, and write $v:= v(\,\cdot\,;x,X)$ and $\rho_h := \rho_h(\,\cdot\,;x,X)$. Then, using Definition \ref{Def: Hbar}, the bound \eqref{HLip}, and Theorem \ref{Thm: FEM for v mixed}, we obtain
\begin{align*}
    &\lvert \overline{H}(x,X) - \overline{H}_h(x,X)\rvert \leq  \| H(x,\,\cdot\,,X+D^2 v) - H(x,\,\cdot\,,X+D\rho_h) \|_{L^2(\T^n)} \\
    &\qquad\leq (\sqrt{1-\delta} + \sqrt{n}) \lvert \nabla v - \rho_h\rvert_{H^1(\T^n)} \\
    &\qquad\leq (\sqrt{1-\delta} + \sqrt{n})\left(1+\frac{1 +  \sqrt{1-\delta} }{1 - \sqrt{1-\delta}}   \right)\inf_{w_h\in \tilde S_h} \lvert \nabla v-w_h\rvert_{H^1(\T^n)},
\end{align*}
where we have used that $C_{\tilde \Pi} = \tilde c_{\T^n}(2) = \tilde d_{\T^n}(2) = 1$.
\end{proof}

\appendix
\section{Some bounds on $c(p)$}\label{Sec: App}

\begin{lmm}\label{Lmm: cp leq 1}
Let $\Omega\subset \mathbb R^n$ be a nonempty open set and let $p\in [1,\infty)$. Then,
\begin{align*}
    c(p) := \inf_{v\in W^{2,p}(\Omega)\cap W^{1,p}_0(\Omega)\backslash\{0\}} \frac{\|\Delta v\|_{L^p(\Omega)}}{\|D^2 v\|_{L^p(\Omega)}} \leq 1.
\end{align*}
\end{lmm}

\begin{proof}
We give the proof for $n\ge 2$ since the case $n=1$ is trivial. Without loss of generality, assume $0\in \Omega$. We fix a small cylinder $Q :=(-r,r)\times U \Subset \Omega$, where $U\subset \mathbb R^{n-1}$ is a nonempty open set and $r>0$. Let $\phi\in C_c^\infty((-1,1)) \backslash \{0\}$ and $\psi\in C_c^\infty(U)\backslash\{0\}$. For $k > \frac{1}{r}$, we write $x'=(x_2,\dots,x_n)$ and consider the function $v_k\in C_c^\infty(\Omega)\backslash\{0\}$ given by $v_k(x):=k^{-2} \phi(kx_1)\psi(x')$ for $x = (x_1,x')\in Q$, extended by zero. For $x = (x_1,x')\in Q$, we compute $\partial_{11}^2 v_k(x)=\phi''(kx_1)\psi(x')$ and $\Delta v_k(x) = \phi''(kx_1)\psi(x')
+k^{-2} \phi(kx_1)\Delta\psi(x')$. As $\|\phi''\|_{L^p(-1,1)}\|\psi\|_{L^p(U)} > 0$, we obtain
\begin{align*}
    c(p)\leq \frac{\|\Delta v_k\|_{L^p(\Omega)}}{\|\partial_{11}^2 v_k\|_{L^p(\Omega)}}
    \leq \frac{k^{-1/p}(\|\phi''\|_{L^p(-1,1)}\|\psi\|_{L^p(U)} + k^{-2}\|\phi\|_{L^p(-1,1)}\|\Delta\psi\|_{L^p(U)})}{k^{-1/p}\|\phi''\|_{L^p(-1,1)}\|\psi\|_{L^p(U)}} \rightarrow  1
\end{align*}
as $k\rightarrow \infty$.
\end{proof}

\begin{lmm}\label{Lmm: cp Om leq Tn}
Let $\Omega\subset \mathbb R^n$ be a nonempty open set and let $p\in [1,\infty)$. Then,
\begin{align}\label{Lmm 52 claim}
    c(p)\leq \inf_{v\in W^{2,p}(\T^n)\cap L^{p}_0(\T^n)\backslash\{0\}} \frac{\|\Delta v\|_{L^p(\T^n)}}{\|D^2 v\|_{L^p(\T^n)}}.
\end{align}
\end{lmm}

\begin{proof} 
Fix $\eta\in C_c^\infty(\Omega)\backslash\{0\}$ and $v\in C^\infty(\T^n)\cap L^p_0(\T^n)\backslash\{0\}$. We identify $v$ with its $\mathbb Z^n$-periodic extension to $\R^n$. For $k\in \N$, we introduce $u_k\in C_c^\infty(\Omega)$ given by
\begin{align*}
    u_k(x):=k^{-2}\eta(x)v(kx)\quad\text{for }x\in \Omega.
\end{align*}
First, we compute
\begin{align*}
D^2 u_k &= \eta\, D^2v(k\,\cdot) + k^{-1}\left(\nabla \eta\otimes \nabla v(k\,\cdot)
+ \nabla v(k\,\cdot)\otimes \nabla\eta \right)+k^{-2}v(k\,\cdot)D^2\eta,\\
\Delta u_k &= \eta\, \Delta v(k\,\cdot)+2k^{-1}\nabla\eta \cdot \nabla v(k\,\cdot) +
k^{-2}v(k\,\cdot)\Delta\eta,
\end{align*}
and use $\eta\in C_c^\infty(\Omega)$ and $v\in C^\infty(\T^n)$ to see that
\begin{align}\label{eq:laplacian-error}
\left\|D^2u_k-\eta\, D^2v(k\,\cdot)\right\|_{L^p(\Omega)} + \left\|\Delta u_k-\eta\,\Delta v(k\,\cdot)\right\|_{L^p(\Omega)} \xrightarrow[k\rightarrow \infty]{} 0.
\end{align}
Next, since $\lim_{k\rightarrow\infty}\int_{\R^n} g(k\,\cdot)h = \int_{\T^n}g\int_{\R^n}h$ for $(g,h)\in C(\T^n)\times  L^1(\R^n)$, we find
\begin{align}\label{eq:laplacian-average}
\begin{split}
\|\eta\, D^2v(k\,\cdot)\|_{L^p(\Omega)}&\xrightarrow[k\rightarrow \infty]{}\|\eta\|_{L^p(\Omega)}\|D^2 v\|_{L^p(\T^n)}\neq 0,\\ \|\eta\,\Delta v(k\,\cdot)\|_{L^p(\Omega)}&\xrightarrow[k\rightarrow \infty]{}\|\eta\|_{L^p(\Omega)}\|\Delta v\|_{L^p(\T^n)}.
\end{split}
\end{align}
Since $u_k \in W^{2,p}(\Omega)\cap W^{1,p}_0(\Omega)\backslash\{0\}$, we can combine \eqref{eq:laplacian-error} and \eqref{eq:laplacian-average} to find that
\begin{align*}
c(p) \leq \frac{\|\Delta u_k\|_{L^p(\Omega)}}{\|D^2u_k\|_{L^p(\Omega)}}
\xrightarrow[k\rightarrow \infty]{} \frac{\|\eta\|_{L^p(\Omega)}\|\Delta v\|_{L^p(\T^n)}}{\|\eta\|_{L^p(\Omega)}\|D^2v\|_{L^p(\T^n)}} = \frac{\|\Delta v\|_{L^p(\T^n)}}{\|D^2v\|_{L^p(\T^n)}}.
\end{align*}

Altogether, we conclude that $c(p)\leq\frac{\|\Delta v\|_{L^p(\T^n)}}{\|D^2 v\|_{L^p(\T^n)}}$ for all $v\in C^\infty(\T^n)\cap L^p_0(\T^n)\backslash\{0\}$. Finally, since $C^\infty(\T^n)\cap L^p_0(\T^n)$ is dense in $W^{2,p}(\T^n)\cap L^p_0(\T^n)$, this implies \eqref{Lmm 52 claim}.
\end{proof}

\section*{AI Statement}

We started this project in February 2025 and developed all results on our own. With the recent developments in AI, we decided to use ChatGPT to explore whether explicit formulas for the constants $c_{\T^n}(p),\xi(p),\tilde c_{\T^n}(p),\tilde d_{\T^n}(q)$ are known in the literature. This led us to \cite{AAB26,BSV13} and \cite{SS23} that yield the formulas for $c_{\T^2}(p)$ in Remark \ref{Rk: expl c} and $\xi(p)$ in Lemma \ref{Lmm: xi expl form}, respectively.

\bibliographystyle{plain}
\bibliography{ref_ST_Isaacs}

\end{document}